\documentclass[11pt, a4paper, twoside,leqno]{amsart}
\usepackage[centering, totalwidth = 420pt, totalheight = 610pt]{geometry}
\usepackage{amssymb, amsmath, amsthm, enumerate, microtype, stmaryrd, url,mathpartir}
\usepackage[latin1]{inputenc}
\usepackage[dvips, arrow, matrix, tips, curve]{xy}
\usepackage{color}
\definecolor{darkgreen}{rgb}{0,0.45,0}
\usepackage[pagebackref,colorlinks,citecolor=darkgreen,linkcolor=darkgreen]{hyperref}

\DeclareMathOperator{\Id}{Id}

\newcommand{\cat}[1]{\mathbf{#1}}

\newcommand{\op}{\mathrm{op}}
\newcommand{\id}{\mathrm{id}}
\newcommand{\thg}{{\mathord{\text{--}}}}
\newcommand{\wo}{\mathrel{\boxslash}}

\newcommand{\abs}[1]{{\left|{#1}\right|}}

\newcommand{\set}[2]{\left\{\,#1 \ \vrule\  #2\,\right\}}

\newcommand{\defeq}{\mathrel{\mathop:}=}

\newcommand{\cd}[2][]{\vcenter{\hbox{\xymatrix#1{#2}}}}

\newcommand{\C}{{\mathcal C}}

\newcommand{\E}{{\mathcal E}}

\newcommand{\I}{{\mathcal I}}
\newcommand{\J}{{\mathcal J}}

\newcommand{\ELL}{{\mathcal L}}

\newcommand{\R}{{\mathcal R}}

\newcommand{\xtor}[1]{\cdl[@1]{{} \ar[r]|-{\object@{|}}^{#1} & {}}}

\makeatletter

\def\hookleftarrowfill@{\arrowfill@\leftarrow\relbar{\relbar\joinrel\rhook}}
\def\twoheadleftarrowfill@{\arrowfill@\twoheadleftarrow\relbar\relbar}
\def\leftbararrowfill@{\arrowdoublefill@{\leftarrow\mkern-5mu}\relbar\mapstochar\relbar\relbar}
\def\Leftbararrowfill@{\arrowdoublefill@{\Leftarrow\mkern-2mu}\Relbar\Mapstochar\Relbar\Relbar}
\def\leftringarrowfill@{\arrowdoublefill@{\leftarrow\mkern-3mu}\relbar{\mkern-3mu\circ\mkern-2mu}\relbar\relbar}
\def\lefttriarrowfill@{\arrowfill@{\mathrel\triangleleft\mkern0.5mu\joinrel\relbar}\relbar\relbar}
\def\Lefttriarrowfill@{\arrowfill@{\mathrel\triangleleft\mkern1mu\joinrel\Relbar}\Relbar\Relbar}

\def\hookrightarrowfill@{\arrowfill@{\lhook\joinrel\relbar}\relbar\rightarrow}
\def\twoheadrightarrowfill@{\arrowfill@\relbar\relbar\twoheadrightarrow}
\def\rightbararrowfill@{\arrowdoublefill@{\relbar\mkern-0.5mu}\relbar\mapstochar\relbar\rightarrow}
\def\Rightbararrowfill@{\arrowdoublefill@{\Relbar\mkern-2mu}\Relbar\Mapstochar\Relbar\Rightarrow}
\def\rightringarrowfill@{\arrowdoublefill@\relbar\relbar{\mkern-2mu\circ\mkern-3mu}\relbar{\mkern-3mu\rightarrow}}
\def\righttriarrowfill@{\arrowfill@\relbar\relbar{\relbar\joinrel\mkern0.5mu\mathrel\triangleright}}
\def\Righttriarrowfill@{\arrowfill@\Relbar\Relbar{\Relbar\joinrel\mkern1mu\mathrel\triangleright}}

\def\leftrightarrowfill@{\arrowfill@\leftarrow\relbar\rightarrow}
\def\mapstofill@{\arrowfill@{\mapstochar\relbar}\relbar\rightarrow}

\renewcommand*\xleftarrow[2][]{\ext@arrow 20{20}0\leftarrowfill@{#1}{#2}}
\providecommand*\xLeftarrow[2][]{\ext@arrow 60{22}0{\Leftarrowfill@}{#1}{#2}}
\providecommand*\xhookleftarrow[2][]{\ext@arrow 10{20}0\hookleftarrowfill@{#1}{#2}}
\providecommand*\xtwoheadleftarrow[2][]{\ext@arrow 60{20}0\twoheadleftarrowfill@{#1}{#2}}
\providecommand*\xleftbararrow[2][]{\ext@arrow 10{22}0\leftbararrowfill@{#1}{#2}}
\providecommand*\xLeftbararrow[2][]{\ext@arrow 50{24}0\Leftbararrowfill@{#1}{#2}}
\providecommand*\xleftringarrow[2][]{\ext@arrow 10{26}0\leftringarrowfill@{#1}{#2}}
\providecommand*\xlefttriarrow[2][]{\ext@arrow 80{24}0\lefttriarrowfill@{#1}{#2}}
\providecommand*\xLefttriarrow[2][]{\ext@arrow 80{24}0\Lefttriarrowfill@{#1}{#2}}

\renewcommand*\xrightarrow[2][]{\ext@arrow 01{20}0\rightarrowfill@{#1}{#2}}
\providecommand*\xRightarrow[2][]{\ext@arrow 04{22}0{\Rightarrowfill@}{#1}{#2}}
\providecommand*\xhookrightarrow[2][]{\ext@arrow 00{20}0\hookrightarrowfill@{#1}{#2}}
\providecommand*\xtwoheadrightarrow[2][]{\ext@arrow 03{20}0\twoheadrightarrowfill@{#1}{#2}}
\providecommand*\xrightbararrow[2][]{\ext@arrow 01{22}0\rightbararrowfill@{#1}{#2}}
\providecommand*\xRightbararrow[2][]{\ext@arrow 04{24}0\Rightbararrowfill@{#1}{#2}}
\providecommand*\xrightringarrow[2][]{\ext@arrow 01{26}0\rightringarrowfill@{#1}{#2}}
\providecommand*\xrighttriarrow[2][]{\ext@arrow 07{24}0\righttriarrowfill@{#1}{#2}}
\providecommand*\xRighttriarrow[2][]{\ext@arrow 07{24}0\Righttriarrowfill@{#1}{#2}}

\providecommand*\xmapsto[2][]{\ext@arrow 01{20}0\mapstofill@{#1}{#2}}
\providecommand*\xleftrightarrow[2][]{\ext@arrow 10{22}0\leftrightarrowfill@{#1}{#2}}
\providecommand*\xLeftrightarrow[2][]{\ext@arrow 10{27}0{\Leftrightarrowfill@}{#1}{#2}}

\makeatother

\newcommand{\twocong}[2][0.5]{\ar@{}[#2] \save ?(#1)*{\cong}\restore}
\newcommand{\twoeq}[2][0.5]{\ar@{}[#2] \save ?(#1)*{=}\restore}
\newcommand{\rtwocell}[3][0.5]{\ar@{}[#2] \ar@{=>}?(#1)+/l 0.2cm/;?(#1)+/r 0.2cm/^{#3}}
\newcommand{\ltwocell}[3][0.5]{\ar@{}[#2] \ar@{=>}?(#1)+/r 0.2cm/;?(#1)+/l 0.2cm/^{#3}}
\newcommand{\ltwocello}[3][0.5]{\ar@{}[#2] \ar@{=>}?(#1)+/r 0.2cm/;?(#1)+/l 0.2cm/_{#3}}
\newcommand{\dtwocell}[3][0.5]{\ar@{}[#2] \ar@{=>}?(#1)+/u  0.2cm/;?(#1)+/d 0.2cm/^{#3}}
\newcommand{\dltwocell}[3][0.5]{\ar@{}[#2] \ar@{=>}?(#1)+/ur  0.2cm/;?(#1)+/dl 0.2cm/^{#3}}
\newcommand{\drtwocell}[3][0.5]{\ar@{}[#2] \ar@{=>}?(#1)+/ul  0.2cm/;?(#1)+/dr 0.2cm/^{#3}}
\newcommand{\dthreecell}[3][0.5]{\ar@{}[#2] \ar@3{->}?(#1)+/u  0.2cm/;?(#1)+/d 0.2cm/^{#3}}
\newcommand{\utwocell}[3][0.5]{\ar@{}[#2] \ar@{=>}?(#1)+/d 0.2cm/;?(#1)+/u 0.2cm/_{#3}}
\newcommand{\dtwocelltarg}[3][0.5]{\ar@{}#2 \ar@{=>}?(#1)+/u  0.2cm/;?(#1)+/d 0.2cm/^{#3}}
\newcommand{\utwocelltarg}[3][0.5]{\ar@{}#2 \ar@{=>}?(#1)+/d  0.2cm/;?(#1)+/u 0.2cm/_{#3}}

\newdir{(}{{}*!<0em,-.14em>-\cir<.14em>{l^r}}
\newdir{ (}{{}*!/-5pt/\dir{(}}
\newdir{ >}{{}*!/-5pt/\dir{>}}

\newtheorem{Thm}{Theorem}[subsection]

\newtheorem{Prop}[Thm]{Proposition}
\newtheorem{Cor}[Thm]{Corollary}
\newtheorem{Lemma}[Thm]{Lemma}

\theoremstyle{definition}

\newtheorem{Defn}[Thm]{Definition}
\newtheorem{Not}[Thm]{Notation}

\newtheorem{Ex}[Thm]{Example}

\newtheorem{Rk}[Thm]{Remark}

\newcommand{\ty}{\mathsf{type}}
\renewcommand{\c}{,\,\,}

\renewcommand{\r}{\mathrm{r}}
\renewcommand{\J}{\mathrm{J}}

\begin{document}
\title[Topological and simplicial models of identity types]{Topological and simplicial\\models of identity types}
\author{Benno van den Berg}
\address{Technische Universit\"at Darmstadt, Fachbereich Mathematik, Schlo\ss gartenstra\ss e 7, 64289 Darmstadt, Germany}
\email{berg@mathematik.tu-darmstadt.de}
\author{Richard Garner}
\address{Department of Computing, Macquarie University, Sydney NSW 2109, Australia}
\email{richard.garner@mq.edu.au} \begin{abstract}
In this paper we construct new categorical models for the identity types of
Martin-L\"of type theory, in the categories $\cat{Top}$ of topological spaces
and $\cat{SSet}$ of simplicial sets. We do so building on earlier work of
Awodey and Warren, which has suggested that a suitable environment for the
interpretation of identity types should be a category equipped with a weak
factorisation system in the sense of Bousfield--Quillen. It turns out that this
is not quite enough for a sound model, due to some subtle coherence issues
concerned with stability under substitution; and so our first task is to
introduce a slightly richer structure---which we call a
\emph{homotopy-theoretic model of identity types}---and to prove that this is
sufficient for a sound interpretation.

Now, although both $\cat{Top}$ and $\cat{SSet}$ are categories endowed with a
weak factorisation system---and indeed, an entire Quillen model
structure---exhibiting the additional structure required for a
homotopy-theoretic model is quite hard to do. However, the categories we are
interested in share a number of common features, and abstracting these leads us
to introduce the notion of a \emph{path object category}. This is a relatively
simple axiomatic framework, which is nonetheless sufficiently strong to allow
the construction of homotopy-theoretic models. Now by exhibiting suitable path
object structures on $\cat{Top}$ and $\cat{SSet}$, we endow those categories
with the structure of a homotopy-theoretic model: and, in this way, obtain the
desired topological and simplicial models of identity types.
\end{abstract}
\thanks{The second author acknowledges the support of the Centre for Australian
Category Theory.}

\maketitle

\newcommand{\Ty}{\mathrm{Ty}}
\section{Introduction}
\looseness=-1 Recently, there have been a number of interesting developments in
the categorical semantics of intensional Martin-L\"of type theory, with work
such
as~\cite{Awodey2008Homotopy,Gambino2008identity,Garner2008Two-dimensional,Garner2008Types,Lumsdaine2009Weak,Warren2008Homotopy}
establishing links between type theory, abstract homotopy theory and
higher-dimensional category theory. All of this work can be seen as an
elaboration of the following basic idea: that in Martin-L\"of type theory, a
type $A$ is analogous to a topological space; elements $a, b \in A$ to points
of that space; and elements of an identity type $p, q \in \Id_A(a,b)$ to
\emph{paths} or \emph{homotopies} $p, q \colon a \to b$ in $A$. This article is
a further development of this theme; its goal is to construct models of the
identity types in categories whose objects have a suitably topological nature
to them---in particular, in the categories of topological spaces and of
simplicial sets.

One popular approach to articulating the topological nature of a category is to
equip it with a model structure in the sense of~\cite{Quillen1967Homotopical}.
It is shown in~\cite{Awodey2008Homotopy} that such a model structure is a
suitable environment for the interpretation of the identity types of
Martin-L\"of type theory; the main point being that we may fruitfully interpret
dependent types $(x \in \Gamma)\, A(x)$ by \emph{fibrations} $A \to \Gamma$ in
the model structure. In fact, a model structure is somewhat more than one
needs: it is comprised of two \emph{weak factorisation
systems}~\cite{Bousfield1977Constructions} interacting in a well-behaved
manner, but as Awodey and Warren point out, in modelling type theory only one
of these weak factorisation systems plays a role.

Now, it is certainly the case that the categories of topological spaces and of
simplicial sets carry well-understood Quillen model structures; and so one
might expect that we could construct models of identity types in them by a
direct appeal to Awodey and Warren's results. In fact, this is not the case due
to a crucial detail---which we have so far elided---concerning the
\emph{stability under substitution} of the identity type structure. In the
categorical interpretation described by Awodey and Warren, substitution is
modelled by pullback, whilst the identity type structure is obtained by
choosing certain pieces of data whose existence is assured by the given weak
factorisation system (henceforth w.f.s.). Thus to ask for the identity type
structure to be stable under substitution is to ask for these choices of data
to be suitably stable under pullback, something which in general is rather hard
to arrange.

A finer analysis of this stability problem is given
in~\cite{Warren2008Homotopy}, which reveals two distinct aspects to it. The
first concerns the stability under substitution of the identity type itself and
of its introduction rule. For many examples derived from w.f.s.'s, including
those studied in~\cite{Warren2008Homotopy} and those studied here, it is
possible---though by no means trivial---to obtain this stability by choosing
one's data with sufficient care. However, the second aspect to the stability
problem is more troublesome. It concerns the stability of the identity type's
elimination and computation rules: and here we know rather few examples of
w.f.s.'s for which the requisite data may be chosen in a suitably coherent
manner.

The first main contribution of this paper is to describe a general solution to
the stability problem for homotopy-theoretic semantics. The key idea is to work
with a mild ``algebraisation'' of the notion of w.f.s.---which we term a
\emph{cloven w.f.s.}---in which certain of the data whose existence is merely
assured by the definition of w.f.s.\ are now provided as part of the structure.
In this setting, we may model dependent types not by fibrations, but rather by
\emph{cloven} fibrations: the difference being that whereas being a fibration
is a \emph{property} of a map, being a cloven fibration is \emph{structure} on
it. This extra structure will turn out to be just what we need to determine
canonical, pullback-stable choices of interpretation for the identity type
elimination rule. We crystallise this idea by introducing a notion of
\emph{homotopy-theoretic model of type theory}---this being a category equipped
with a cloven w.f.s.\ and suitable extra data related to that w.f.s.---and
proving our first main result, that every homotopy-theoretic model admits a
sound interpretation of the identity types of Martin-L\"of type theory.

Now, any reasonable w.f.s.\ on a category may be equipped with the structure of
a cloven w.f.s.: but it is by no means always the case that this cloven w.f.s.\
can be made part of a homotopy-theoretic model of identity types. This is
because the ``extra data'' required to do so---namely, a pullback-stable choice
of factorisations for diagonal morphisms $X \to X \times_Y X$---is not
something we  expect to exist in general. The second main contribution of this
paper is to describe a simple and widely applicable axiomatic framework---that
of a \emph{path object category}---within which one may construct cloven
w.f.s.'s which do carry this extra data. The fundamental axiom of our framework
is one asserting the existence for every object $X$ of the given category of a
``path object'' $MX$ providing an internal category structure $MX
\rightrightarrows X$ on $X$. From this we obtain a cloven w.f.s.\ whose
fibrations are the maps with the path-lifting property with respect to this
notion of path. The reason that this cloven w.f.s.\ may be equipped with the
extra data required for a homotopy-theoretic model is that the notion of path
object category is ``stable under slicing'': which is to say that any slice of
a path object category is a path object category, and that any pullback functor
between slices preserves the structure. We may thereby construct
pullback-stable factorisations of diagonals by factorising $X \to X \times_Y X$
using the path object structure on the slice over $Y$. We encapsulate this in
the second main result of the paper, which shows that every path object
category gives rise to a homotopy-theoretic model of type theory.

The third main contribution of our paper is to exhibit a number of instances of
the notion of path object category, hence obtaining a range of different models
of identity types. Some of the models we obtain are already known, such as the
groupoid model of~\cite{Hofmann1998groupoid}, and the chain complex model
of~\cite{Warren2008Homotopy}. More generally, our framework allows us to
capture a class of models described in~\cite{Warren2008Homotopy} whose
structure is determined by an \emph{interval object} in a category. However,
beyond these classes of known models, we obtain two important new ones: the
first in the category of topological spaces, and the second in the category of
simplicial sets. Let us also note that Jaap van Oosten has communicated the
existence of a further instance of our axiomatic framework in the
\emph{effective topos} of~\cite{Hyland1982effective}; this extends his work
in~\cite{Oosten2010Notion}.

It is as well to point out also what we do \emph{not} achieve in this article.
The categorical models that one builds from the syntax of Martin-L\"of type
theory turn out to be neither homotopy-theoretic models nor path object
categories; and so there can be no hope of a completeness result for
intensional type theory with respect to semantics valued in either kind of
model. The reason for this is a certain strictness present in these structures,
necessary for the arguments we make, but not present in the syntax. This same
strictness also has ramifications for the simplicial and topological models we
construct. In both cases, the obvious choices of path object---given in the
topological case by the assignation $X \mapsto X^{[0,1]}$ and in the simplicial
case by $X \mapsto X^{\Delta^1}$, where $\Delta^1$ denotes the simplicial
interval---are insufficiently strict to yield a path object structure,
necessitating a subtler model construction using the notion of (topological or
simplicial) \emph{Moore path}. It remains an open problem as to whether there
is a more refined notion of homotopy-theoretic model which admits both the
syntax and the ``naive'' simplicial and topological interpretations as
examples.

The plan of the paper is as follows. We begin in Section~\ref{idmodels} by
recalling the syntax and semantics of the type theory we will be concerned
with. Then in Section~\ref{htmodels}, we introduce the notion of a
homotopy-theoretic model of identity types, and prove that every such model
admits a sound interpretation of our type theory. Next, in
Section~\ref{catframework}, we introduce the axiomatic structure of a path
object category; in Section~\ref{exs}, we give a number of examples of path
object categories, including the category of topological spaces, and the
category of simplicial sets; and in Section~\ref{typecatstruct}, we show that
every path object category may be made into a homotopy-theoretic model of
identity types, and so admits a sound interpretation of our type theory.
Finally, Section~\ref{simplicialexample} fills in the combinatorial details of
the construction of the path object category of simplicial sets.


\section{Syntax and semantics of dependent type theory}\label{idmodels}
In this Section, we gather together the required type-theoretic background:
firstly describing the syntax of the type theory with which we shall work, and
then the corresponding notion of categorical model.
\subsection{Intensional type theory}
By \emph{intensional Martin-L\"of type theory}, we mean the logical calculus
set out in Part~I of~\cite{Nordstrom1990Programming}. Our concern in the
present paper will be with the fragment of this theory containing only the
basic structural rules together with the rules for the \emph{identity types}.
We now summarise this calculus. It has four basic forms of judgement: $A \ \ty$
(``$A$ is a type''); $a : A$ (``$a$ is an element of the type $A$''); $A = B \
\ty$ (``$A$ and $B$ are definitionally equal types''); and $a = b : A$ (``$a$
and $b$ are definitionally equal elements of the type $A$''). These judgements
may be made either absolutely, or relative to a context $\Gamma$ of
assumptions, in which case we write them as
\begin{equation*}
  \Gamma \ \vdash \ A\ \ty\text, \qquad
  \Gamma \ \vdash \ a : A\text, \qquad
  \Gamma \ \vdash \ A = B\ \ty \qquad \text{and} \qquad
  \Gamma \ \vdash \ a = b: A
\end{equation*}
respectively. Here, a \emph{context} is a list $\Gamma = x_1 : A_1\c x_2 :
A_2\c \dots\c x_n : A_{n-1}$, wherein each $A_i$ is a type relative to the
context $x_1 : A_1\c \dots\c x_{i-1} : A_{i-1}$. There are now some rather
natural requirements for well-formed judgements: in order to assert that $a :
A$ we must first know that $A \ \ty$; to assert that $A = B \ \ty$ we must
first know that $A \ \ty$ and $B \ \ty$; and so on. We specify intensional
Martin-L\"of type theory as a collection of inference rules over these forms of
judgement. Firstly we have the \emph{equality rules}, which assert that the two
judgement forms $A = B \ \ty$ and $a = b : A$ are congruences with respect to
all the other operations of the theory; then we have the \emph{structural
rules}, which deal with weakening, contraction, exchange and substitution; and
finally, the \emph{logical rules}, which specify the type-formers of our
theory, together with their introduction, elimination and computation rules.
The only logical rules we consider in this paper are those for the identity
types, which we list in Figure~\ref{fig1}. We commit the usual abuse of
notation in leaving implicit an ambient context $\Gamma$ common to the
premisses and conclusions of each rule, and omitting the rules expressing
stability under substitution in this ambient context. Let us remark also that
in the rules $\Id\textsc{-elim}$ and $\Id\textsc{-comp}$ we allow the type $C$
over which elimination is occurring to depend upon an additional contextual
parameter $\Delta$. We refer to these forms of the rules as the \emph{strong}
computation and elimination rules. Were we to add $\Pi$-types (dependent
products) to our calculus, then these rules would be equivalent to the usual
ones, without the extra parameter $\Delta$; however, in the absence of
$\Pi$-types, this extra parameter is essential to derive all but the most basic
properties of the identity type.

\begin{figure}
\begin{equation*}
    \inferrule*[right=$\Id$-form;]{A\ \ty \\ a, b : A}{\Id_A(a, b) \ \ty} \qquad
    \inferrule*[right=$\Id$-intro;]{A\ \ty \\ a : A}{\r(a) : \Id_A(a, a)}
\end{equation*}
\ \begin{equation*}
\inferrule*[right=$\Id$-elim;]{\big(x, y : A \c p : \Id_A(x, y) \c \Delta(x,y,p)\big)\  C(x, y, p) \ \ty\\
  x : A \c \Delta(x, x, \r(x))\ \vdash \  d(x) : C(x, x, \r(x))\\
  a, b : A \\ p : \Id_A(a, b)}
  {\Delta(a,b,p)\ \vdash \ \J_{d}(a, b, p) : C(a, b, p)}
\end{equation*}
\
\begin{equation*}
\inferrule*[right=$\Id$-comp.]{x, y : A \c p : \Id_A(x, y) \c \Delta(x,y,p)\ \vdash \   C(x, y, p) \ \ty\\
  x : A \c \Delta(x, x, \r(x))\ \vdash \ d(x) : C(x, x, \r(x))\\
  a : A}
  {\Delta(a,a,\r(a))\ \vdash \ \J_{d}(a,a,\r(a)) = d(a) : C(a, a, \r(a))}
\end{equation*}
\caption{Identity type rules}\label{fig1}
\end{figure}

\subsection{Models of type theory}
We now give a suitable notion of \emph{categorical model} for the dependent
type theory we have just described. There are a number of essentially
equivalent notions of categorical model we could use (see, for example,
\cite{Cartmell1986Generalised,Jacobs1993Comprehension,Pitts2000Categorical});
of these, we have chosen Pitts' \emph{type categories} since they minimise the
amount of data required to construct a model, but still admit a precise
soundness and completeness result. We first recall
from~\cite{Pitts2000Categorical} the basic definition.
\begin{Defn}
A \emph{type category} is given by:
\begin{itemize}
\item A category $\C$ of \emph{contexts}.
\item For each $\Gamma \in \C$, a collection $\Ty(\Gamma)$ of \emph{types}
    in context $\Gamma$.
\item For each $A \in \Ty(\Gamma)$, an \emph{extended context} $\Gamma.A
    \in \C$ and a \emph{dependent projection} $\pi_A \colon \Gamma.A \to
    \Gamma$.
\item For each $f \colon \Delta \to \Gamma$ in $\C$ and $A \in
    \Ty(\Gamma)$, a type $A[f] \in \Ty(\Delta)$ and a morphism $f^+ \colon
    \Delta. A[f] \to \Gamma.A$ making the following square into a pullback:
\begin{equation}\label{necpb}
    \cd{
      \Delta. A[f] \ar[r]^-{f^+} \ar[d]_{\pi_{A[f]}} & \Gamma.A \ar[d]^{\pi_A}
\\
      \Delta  \ar[r]_{f} & \Gamma\rlap{ .}
    }
\end{equation}
\end{itemize}
A type category is said to be \emph{split} if it satisfies the coherence axioms
\begin{equation}\label{cohlaws}
    A[\id_\Gamma] = A\text, \qquad     A[f g] = A[f][g]\text, \qquad
    (\id_\Gamma)^+ = \id_{A.\Gamma}\text, \qquad (f g)^+ = f^+g^+\text.
\end{equation}
\end{Defn}
\begin{Rk}\label{coherenceremark1}
In~\cite{Pitts2000Categorical}, the coherence axioms of~\eqref{cohlaws} are
taken as part of the definition of a type category. We do not do so here due to
the nature of the type categories we wish to construct: in them, the types over
$\Gamma$ will be (structured) maps $X \to \Gamma$ and the operation of type
substitution will be given by pullback, an operation which is rarely functorial
on the nose. However, as is pointed out in~\cite{Hofmann1995interpretation},
without the coherence laws of~\eqref{cohlaws}, we cannot obtain a sound
interpretation of the structural axioms of a dependent type theory. The main
result of that paper allows us to overcome this: when translated into our
language, it says that any type category may be replaced by a split type
category which is equivalent to it in a suitable $2$-category of type
categories.
\end{Rk}
 We now describe the additional structure required on a type category for it to model identity types. First we introduce some notation. Given $A, B \in
\Ty(\Gamma)$, we write $B^+$ as an abbreviation for $B[\pi_A] \in
\Ty(\Gamma.A)$. Thus we have a pullback square
\begin{equation*}
    \cd{
      \Gamma.A.B^+ \ar[r]^-{(\pi_A)^+} \ar[d]_{\pi_{B^+}} & \Gamma.B
\ar[d]^{\pi_B} \\
      \Gamma.A \ar[r]_{\pi_A} & \Gamma
    }\ \text.
\end{equation*}
In particular, when $A = B$, the universal property of this pullback induces a
diagonal morphism $\delta_A \colon \Gamma.A \to \Gamma.A.A^+$ satisfying
$\pi_{A^+} . \delta_A = (\pi_A)^+ . \delta_A = \id_{\Gamma.A}$.
\begin{Defn}\label{idtypesdefn}
A type category has \emph{identity types} if there are given:
\begin{itemize}
\item For each $A \in \Ty(\Gamma)$, a type $\Id_A \in \Ty(\Gamma.A.A^+)$;
\item For each $A \in \Ty(\Gamma)$, a morphism
    $r_A \colon \Gamma.A \to \Gamma.A.A^+.\Id_A$
with $\pi_{\Id_A}.r_a = \delta_A$;
\item For each $C \in \Ty(\Gamma.A.A^+.\Id_A)$ and commutative diagram
\begin{equation}\label{idelim}
    \cd{
      \Gamma.A \ar[r]^-{d} \ar[d]_{r_A} & \Gamma.A.A^+.\Id_A.C \ar[d]^{\pi_C} \\
      \Gamma.A.A^+.\Id_A \ar[r]_{\id} & \Gamma.A.A^+.\Id_A
    }
\end{equation}
a diagonal filler $J(C, d)$ making both triangles commute.
\end{itemize}
We further require that this structure should be stable under substitution.
Thus, for every morphism $f \colon \Delta \to \Gamma$ in $\C$, we require that
$\Id_A[f^{++}] = \Id_{A[f]}$, and that the following two squares should
commute:
\begin{equation}\label{subststabdiags}
   \cd{
    \Delta.A[f] \ar[r]^-{r_{A[f]}} \ar[d]_{f^{+}} &
    \Delta.A[f].A[f]^+.\Id_{A[f]} \ar[d]^{f^{+++}} \\
    \Gamma.A \ar[r]_-{r_A} &
    \Gamma.A.A^+.\Id_{A}
  }
\end{equation}
\begin{equation}
  \cd[@C+3em]{
    \llap{$\Delta.A[f].$}A[f]^+.\Id_{A[f]} \ar[d]_{f^{+++}}
\ar[r]^-{J(C[f],d[f])} &
    \Delta.A[f].A[f]^+.\rlap{$\Id_{A[f]}.C[f^{+++}]$} \ar[d]^{f^{++++}} \\
    \Gamma.A.A^+.\Id_{A} \ar[r]_-{J(C,d)} &
    \Gamma.A.A^+.\Id_{A}\rlap{$.C$ .}
  }\label{subststabdiags2}
\end{equation}
\end{Defn}
As discussed above, the most appropriate formulation of the identity type rules
in the absence of $\Pi$-types incorporates an extra contextual parameter in the
elimination and computation rules. However, the structure we have just
described captures only the weaker forms in which this contextual parameter is
absent. Let us therefore formulate the strong computation and elimination rules
in our categorical setting.
\begin{Defn}\label{strongids}
A type category has \emph{strong identity types} if for every $A \in
\Ty(\Gamma)$ there are given $\Id_A$ and $r_A$ as above, but now for every
\begin{align*}
B_1 & \in \Ty(\Gamma.A.A^+.\Id_A)\\
& \vdots\\
B_n & \in \Ty(\Gamma.A.A^+.\Id_A.B_1\dots B_{n-1})\\
C & \in \Ty(\Gamma.A.A^+.\Id_A.B_1\dots B_{n-1}.B_n)
\end{align*}
and commutative diagram
\begin{equation}\label{strongiddiag}
    \cd{
      \Gamma.A.\Delta[r_A] \ar[r]^-{d} \ar[d]_{(r_A)^{+\dots+}} &
\Gamma.A.A^+.\Id_A.\Delta.C \ar[d]^{\pi_C} \\
      \Gamma.A.A^+.\Id_A.\Delta \ar[r]_{\id} & \Gamma.A.A^+.\Id_A.\Lambda
    }
\end{equation}
(where we write $\Delta$ as an abbreviation for $B_1 \dots B_n$ in the obvious
way), we are given a diagonal filler $J(\Lambda,C,d)$ making both triangles
commute. We require all this structure to be stable under substitution as in
Definition~\ref{idtypesdefn}.

By a \emph{categorical model of identity types}, we mean a type category with
strong identity types.
\end{Defn}

\begin{Rk}
As discussed in Remark~\ref{coherenceremark1} above, our use of non-split type
categories is justified by the coherence result
of~\cite{Hofmann1995interpretation}, which allows us to replace any non-split
type category by an equivalent split one. For this justification to remain
meaningful in the presence of identity types, it must be the case that a
(strong) identity type structure on a type category induces a corresponding
structure on its strictification. This is indeed the case, as proven
in~\cite[Theorem 2.48]{Warren2008Homotopy}. (Actually, Warren does not consider
the strong identity type rules; but his argument may be modified without
difficulty to cover this case.)
\end{Rk}

\section{Homotopy-theoretic models of identity types}\label{htmodels}\looseness=-1
In this section, we define a notion of \emph{homotopy-theoretic model of
identity types}---building on the work of~\cite{Awodey2008Homotopy}---and prove
our first main result, Theorem~\ref{mainthm1}, which shows that any
homotopy-theoretic model gives rise to a categorical one. The notion of model
described here can be seen as a precise formulation of one that is implicit in
Chapter~3 of~\cite{Warren2008Homotopy}.

\subsection{Interpretation of identity types in a weak factorisation system}
The notion of homotopy-theoretic model which we are going to define is based on
the central idea of~\cite{Awodey2008Homotopy}: that a suitable environment for
the interpretation of identity types is that of a category equipped with a
\emph{weak factorisation system} in the sense
of~\cite{Bousfield1977Constructions}. We begin by recalling the basic
definitions.
\begin{Defn}
A \emph{weak factorisation system} or \emph{w.f.s.}\ $(\ELL, \R)$ on a
category~$\E$ is given by two classes $\ELL$ and~$\R$ of morphisms in~$\E$
which are each closed under retracts when viewed as full subcategories of the
arrow category $\E^\mathbf 2$, and which satisfy the following two axioms:
\begin{enumerate}
\item[(i)] \emph{Factorisation}: each $f \in \E$ may be written as $f = pi$
    where $i \in \ELL$ and $p \in \R$.
\item[(ii)] \emph{Weak orthogonality}: for each $f \in \ELL$ and $g \in
    \R$, we have $f \wo g$,
\end{enumerate}
where to say that $f \wo g$ holds is to say that for each commutative square
\begin{equation}\label{cs}
    \cd{
      U \ar[r]^{h} \ar[d]_{f} &
      X \ar[d]^{g} \\
      V \ar[r]_{k} &
      Y
    }
\end{equation}
we may find a filler $j \colon V \to X$ satisfying $jf = h$ and $gj = k$.
\end{Defn}
Given a category $\E$ equipped with a w.f.s., \cite{Awodey2008Homotopy}
suggests the following method of interpreting identity types in it. One begins
by taking the dependent types over some $\Gamma \in \E$ to be the collection of
$\R$-maps $X \to \Gamma$, with type substitution being given by pullback. Now
given a map $x \colon X \to \Gamma \in \E$ interpreting some dependent type
over $\Gamma$, we may factorise the diagonal $X \to X \times_\Gamma X$~as
\begin{equation}\label{diagfac}
  X \xrightarrow{i_x} I(x) \xrightarrow{j_x} X \times_\Gamma X\ \text,
\end{equation}
where $i_x$ is an $\ELL$-map, and $j_x$ an $\R$-map. Since $j_x$ is an
$\R$-map, it gives rise to a dependent type over $X \times_\Gamma X$, which
will be the interpretation of the identity type on $X$. The introduction rule
for $\Id_X$ will be interpreted by the map $i_x$; whilst to interpret the
elimination and computation rules, we observe that given a diagram
like~\eqref{idelim} in $\E$, the left-hand arrow is in $\ELL$---since it is the
map $i_x$ above---and the right-hand arrow is in $\R$---by definition of
dependent type in the model---so that by weak orthogonality, we have a filler
$J(C,d)$ as required.

\subsection{Cloven weak factorisation systems}
As discussed in the Introduction, when we try to make the above argument
precise we run into a problem of \emph{coherent choice}. The definition of weak
factorisation system demands the existence of certain pieces of
data---factorisations and diagonal fillers---without asking for explicit
choices of these data to be made. In order to model type theory, therefore, we
must first choose factorisations as in~\eqref{diagfac}, and fillers for squares
such as~\eqref{idelim}. However, we cannot make such choices arbitrarily, since
the categorical structure defining the identity types is required to be stable
under substitution, which amounts to requiring that the choices of
factorisations and fillers we make be stable under pullback. As noted in the
Introduction, the two aspects of this requirement---stability of
factorisations, and stability of fillers---are quite different in nature. For
whilst many naturally-occurring weak factorisation systems may be equipped with
stable factorisations~\eqref{diagfac}---as is worked out comprehensively
in~\cite{Warren2008Homotopy}---rather few have been similarly provided with
stable fillers~\eqref{idelim}. A closer analysis of the examples where this has
been possible shows that underlying each of them is a structure richer than
that of a mere w.f.s.\ The following definition is intended to capture the
essence of that extra structure.

\begin{Defn}
A \emph{cloven w.f.s.} on a category $\E$ is given by the following data:
\begin{itemize}
\item For each $f \colon X \to Y$ in $\E$, a choice of factorisation
\begin{equation}\label{factorisation}
    f = X \xrightarrow{\lambda_f} Pf \xrightarrow{\rho_f} Y\ \text;
\end{equation}
\item For each commutative square of the form~\eqref{cs}, a choice of
    diagonal fillers
\begin{equation}\label{diagfiller}
    \cd{
        U \ar[r]^{\lambda_g . h} \ar[d]_{\lambda_f} & Pg \ar[d]^{\rho_g} \\
        Pf \ar[r]_{k . \rho_f} \ar@{.>}[ur]|{P(h, k)} & Y }
\end{equation}
making the assignation $(h,k) \mapsto P(h,k)$ functorial in $(h,k)$\
;\vskip0.5\baselineskip
\item For each $f \colon X \to Y$, choices of fillers $\sigma_f$ and
    $\pi_f$ as indicated:
\begin{equation}\label{sigpi}
    \cd{
      X \ar[r]^{\lambda_{\lambda_f}} \ar[d]_{\lambda_f} &
      P\lambda_f \ar[d]^{\rho_{\lambda_f}} \\
      Pf \ar[r]_{1_{Pf}} \ar@{.>}[ur]_{\sigma_f} &
      Pf
    } \qquad \text{and} \qquad
    \cd{
      Pf \ar[r]^{1_{Pf}} \ar[d]_{\lambda_{\rho_f}} &
      Pf \ar[d]^{\rho_f} \\
      P\rho_f \ar[r]_{\rho_{\rho_f}} \ar@{.>}[ur]_{\pi_f} &
      Y\rlap{ .}
    }
\end{equation}
\end{itemize}
\end{Defn}
To justify the nomenclature, we must show that any cloven w.f.s.\ has an
underlying w.f.s. To do this, we introduce the notion of cloven $\ELL$- and
$\R$-maps in a cloven w.f.s. By a \emph{cloven $\ELL$-map structure} on a
morphism $f \colon X \to Y$ of $\E$, we mean a map $s \colon Y \to Pf$
rendering commutative the diagram
\begin{equation*}
    \cd{
      X \ar[r]^{\lambda_f} \ar[d]_{f} &
      Pf \ar[d]^{\rho_f} \\
      Y \ar[r]_{1_Y} \ar@{.>}[ur]_{s} &
      Y\rlap{ .}
    }
\end{equation*}
We will sometimes express this by saying that $(f,s) \colon X \to Y$ is a
cloven $\ELL$-map. Dually, a \emph{cloven $\R$-map} structure on $f$ is given
by a morphism $p \colon Pf \to X$ rendering commutative the diagram
\begin{equation*}
    \cd{
      X \ar[r]^{1_X} \ar[d]_{\lambda_f} &
      X \ar[d]^{f} \\
      Pf \ar[r]_{\rho_f} \ar@{.>}[ur]_{p} &
      Y\rlap{ .}
    }
\end{equation*}
Again, we may express this by calling $(f, p) \colon X \to Y$ a cloven
$\R$-map.
\begin{Prop}\label{howtogetfillers}
Given a cloven $\ELL$-map $(f, s) \colon U \to V$, a cloven $\R$-map $(g, p)
\colon X \to Y$ and a commutative square of the form~\eqref{cs}, there is a
canonical choice of diagonal filler $j \colon V \to X$ making both induced
triangles in~\eqref{cs} commute.
\end{Prop}
\begin{proof}
We take $j$ to be the composite 
$    V \xrightarrow{s} Pf \xrightarrow{P(h, k)} Pg \xrightarrow{p} X$.
\end{proof}
The following result is now essentially Section 2.4 of~\cite{Rosick'y2002Lax}:
\begin{Cor}\label{underlying}
Every cloven w.f.s.\ has an underlying w.f.s.\ whose two classes of maps are
given by
\begin{align*}
    \ELL & \defeq \set{f \colon A \to B}{\text{there is a cloven $\ELL$-map
structure on $f$}} \\
    \R & \defeq \set{g \colon C \to D}{\text{there is a cloven $\R$-map
structure on $g$}}\ \text.
\end{align*}
\end{Cor}
\begin{proof}
Firstly, it's easy to show that $\ELL$ and $\R$ are closed under retracts.
Secondly, for each $f \colon X \to Y$ we have $\lambda_f \in \ELL$ since
$(\lambda_f, \sigma_f)$ is a cloven $\ELL$-map, and $\rho_f \in \R$ since
$(\rho_f, \pi_f)$ is a cloven $\R$-map; and so we have the factorisation axiom.
Finally, we must show that $f \wo g$ for all $f \in \ELL$ and $g \in \R$. But
to do this we pick a cloven $\ELL$-map structure on $f$ and a cloven $\R$-map
structure on $g$ and then apply the preceding Proposition.
\end{proof}

\subsection{Homotopy-theoretic models of identity types}
Let us now see how the notion of cloven w.f.s.\ allows us to resolve the
problem of coherent choice with regard to fillers for squares
like~\eqref{idelim}. The idea is to refine our previous interpretation by
taking the dependent types over some $\Gamma \in \C$ to be \emph{cloven}
\mbox{$\R$-maps} $X \to \Gamma$. For each such map, we demand the existence of
a factorisation of the diagonal $X \to X \times_\Gamma X$ into a \emph{cloven}
$\ELL$-map followed by a \emph{cloven} $\R$-map; whereupon
Proposition~\ref{howtogetfillers} provides us with canonical choices of fillers
for squares like~\eqref{idelim}. Now by asking that the choices of
factorisation in~\eqref{diagfac} are suitably stable under substitution---as we
do in Definition~\ref{stablepath} below---we may ensure the stability of the
corresponding fillers in~\eqref{idelim} by exploiting a ``naturality'' property
of the liftings provided by Proposition~\ref{howtogetfillers}. In order to
describe this property, we first need a definition.

\begin{Defn}\label{mormaps}
If $(f, s) \colon U \to V$ and $(g, t) \colon X \to Y$ are cloven $\ELL$-maps,
then by a \emph{morphism of cloven $\ELL$-maps} $(f,s) \to (g,t)$ we mean a
commutative square~\eqref{cs}
such that $P(h,k).s = t.k$. We write $\ELL\text-\cat{Map}$ for the category of
cloven $\ELL$-maps and cloven $\ELL$-map morphisms. Dually, we have the notion
of \emph{morphism of cloven $\R$-maps}, giving the arrows of a category
$\R\text-\cat{Map}$.
\end{Defn}

It is now easy to verify the following:
\begin{Prop}\label{naturalfillers}
The choices of filler given by Proposition~\ref{howtogetfillers} are
\emph{natural}, in the sense that precomposing a square like~\eqref{cs} with a
morphism of cloven $\ELL$-maps $(f',s') \to (f,s)$ sends chosen fillers to
chosen fillers, as does postcomposing it with a morphism of cloven $\R$-maps
$(g, p) \to (g', p')$.
\end{Prop}
There is one final point which we have not yet addressed. In the preceding
discussion, we have outlined how we might obtain an interpretation of identity
types in a homotopy-theoretic setting. What we have not discussed is how to
interpret the \emph{strong} elimination and computation rules. Now, to ask for
an interpretation of the strong rules is to ask for coherent choices of
diagonal filler for squares of the form~\eqref{strongiddiag}. In such a square
we know that the arrow $\pi_C$ down the right-hand side is a cloven $\R$-map,
so that if we could show that the map $(r_A)^{+ \cdots +}$ down the left-hand
side was a cloven $\ELL$-map, then we could once again obtain the desired
liftings by applying Proposition~\ref{howtogetfillers}. But observing that
$(r_A)^{+ \cdots +}$ is the pullback of $r_A$ along a composite of dependent
projections, we obtain the desired conclusion whenever our cloven w.f.s.\
satisfies---in a suitably functorial form---the \emph{Frobenius} property, that
the pullback of any cloven $\ELL$-map along a cloven $\R$-map should again be a
cloven $\ELL$-map. Note that this property has been considered before in the
context of Martin-L\"of type theory: see~\cite[Proposition
14]{Gambino2008identity} and \cite[Definition 3.2.1]{Garner2008Types}, for
example.

With this last detail in place, we are now ready to give our notion of
homotopy-theoretic model.
\begin{Defn}\label{stablepath}
Suppose that $\E$ is a finitely complete category equipped with a cloven w.f.s.
\begin{enumerate}[(i)]
\item  A \emph{choice of diagonal factorisations} is an assignation which
    to every cloven $\R$-map $(x,p) \colon X\to \Gamma$ associates a
    factorisation
\begin{equation}\label{assignedfac}
X \xrightarrow{i_x} I(x) \xrightarrow{j_x} X \times_\Gamma X
\end{equation}
of the diagonal $X \to X \times_\Gamma X$, together with a cloven
$\ELL$-map structure on $i_x$ and a cloven $\R$-map structure on $j_x$.
\vskip0.5\baselineskip
\item A choice of diagonal factorisations is \emph{functorial} if the
    assignation of~\eqref{assignedfac} provides the action on objects of a
    functor $\R$-$\cat{Map} \to \R$-$\cat{Map} \times_\E \ELL$-$\cat{Map}$.
    Explicitly, this means that for every morphism of \mbox{$\R$-maps}
    $(f,g) \colon (x,p) \to (y,q)$, there is given an arrow $I(f,g) \colon
    I(x) \to I(y)$, functorially in $(x,p)$, and in such a way that the
    squares
\begin{equation}\label{twosquares}
\cd{
  X \ar[d]_{i_x} \ar[r]^f & Y \ar[d]^{i_y} \\ I(x) \ar[r]_{I(f,g)} & I(y)
} \qquad \text{and} \qquad
\cd[@C+1em]{
  I(x) \ar[d]_{j_x} \ar[r]^-{I(f,g)} & I(y) \ar[d]^{j_y} \\ X \times_\Gamma X \ar[r]_-{f \times_g f} & Y \times_\Delta Y
}
\end{equation}
are morphisms of $\ELL$-maps and of $\R$-maps
respectively.\vskip0.5\baselineskip
\item A functorial choice of diagonal factorisations is \emph{stable} when
	for every morphism of $\R$-maps $(f,g) \colon (x,p) \to (y,q)$ whose
	underlying morphism in $\E^\mathbf 2$ is a pullback square, the
	right-hand square in~\eqref{twosquares} is also a
pullback.\vskip0.5\baselineskip
\item The cloven w.f.s.\ is \emph{Frobenius} if to every pullback square
\begin{equation*}
\cd{
  f^\ast X \ar[r]^{\bar f} \ar[d]_{\bar \imath} & X \ar[d]^i \\
  B \ar[r]_f & A
}
\end{equation*}
with $f$ a cloven $\R$-map and $i$ a cloven $\ELL$-map, we may assign a
cloven $\ELL$-map structure on $\bar \imath$. It is \emph{functorially
Frobenius} if this assignation gives rise to a functor $\R$-$\cat{Map}
\times_\E \ELL$-$\cat{Map} \to \ELL$-$\cat{Map}$.
\end{enumerate}
A \emph{homotopy-theoretic model of identity types} is a finitely complete
category $\E$ equipped with a cloven w.f.s.\ which is functorially Frobenius
and has a stable functorial choice of diagonal factorisations.
\end{Defn}
\begin{Rk}\label{nontrivial}
Note that any cloven w.f.s.\ has an obvious functorial choice of diagonal
factorisations: given an $\R$-map $(x,p) \colon X \to \Gamma$, we factorise the
diagonal morphism $\delta_x \colon X \to X \times_\Gamma X$~as
$(\lambda_{\delta_x}, \rho_{\delta_x})$,
with the cloven $\ELL$-\ and $\R$-\ structures given by $\sigma_{\delta_x}$ and
$\pi_{\delta_x}$. However, this choice is not a particularly useful one for our
purposes, since it is almost never stable. It will be the task of the next
section, and the second main contribution of this paper, to describe a general
structure---that of a \emph{path object category}---from which we may construct
cloven w.f.s.'s which do have a stable functorial choice of diagonal
factorisations.
\end{Rk}
The remainder of this section will be devoted to proving our first main
result:\begin{Thm}\label{mainthm1}A homotopy-theoretic model of identity types
is a categorical model of identity types.
\end{Thm}
We begin by defining the type category associated to a cloven w.f.s.
%
\begin{Prop}\label{typecatconstruct}
Let $\E$ be a finitely complete category equipped with a cloven w.f.s. Then
there is a type category whose category of contexts is $\E$, and whose
collection of types over $\Gamma \in \E$ is the set of cloven $\R$-maps with
codomain $\Gamma$.
\end{Prop}
\begin{proof}
For $A \in \Ty(\Gamma)$ corresponding to a cloven $\R$-map $(x,p) \colon X \to
\Gamma$, we define the extended context $\Gamma.A$ to be $X$ and the dependent
projection $\pi_A \colon \Gamma.A \to \Gamma$ to be $x$. Given further a
morphism $f \colon \Delta \to \Gamma$ of $\E$, we must define a type $A[f] \in
\Ty(\Delta)$ and a morphism $f^+ \colon \Delta.A[f] \to \Gamma.A$. So let the
following be a pullback diagram in $\E$:
\begin{equation*}
\cd{
  Y \ar[r]^{g} \ar[d]_y &
  X \ar[d]^x \\
  \Delta \ar[r]_f &
  \Gamma\rlap{ .}
}
\end{equation*}
Now let $q \colon Py \to Y$ be the morphism induced by the universal property
of pullback in the following diagram:
\begin{equation*}
\cd[@C+2em]{
  Py \ar[r]^-{p.P(g,f)} \ar[d]_{\rho_h} &
  X \ar[d]^x \\
  \Delta \ar[r]_f &
  \Gamma\rlap{ .}
}
\end{equation*}
It's easy to check that this makes $(y,q) \colon Y \to \Delta$ into a cloven
$\R$-map, which we define to be $A[f] \in \Ty(\Delta)$. Now we take $f^+ \colon
\Delta.A[f] \to \Gamma.A$ to be $g$, which clearly makes~\eqref{necpb} into a
pullback as required. Let us observe for future reference that the pair $(g,f)$
determines a morphism of $\R$-maps $(y,q) \to (x,p)$; and that $q$ is in fact
the \emph{unique} cloven $\R$-map structure on $y$ for which this the case.
\end{proof}
We now show that in the presence of a stable functorial choice of diagonal
factorisations, the type category constructed by the preceding Proposition has
identity types which are stable under substitution.
\begin{Prop}\label{typecathasids}
Let $\E$ be a finitely complete category bearing a cloven w.f.s.\ that is
equipped with a stable functorial choice of diagonal factorisations. Then the
associated type category of Proposition~\ref{typecatconstruct} has identity
types.
\end{Prop}
\begin{proof}
Suppose given $A \in \Ty(\Gamma)$. To define $\Id_A$, we observe that the
square
\begin{equation*}
    \cd{
        \Gamma.A.A^+ \ar[r]^-{(\pi_A)^+} \ar[d]_{\pi_{A^+}} & \Gamma.A
\ar[d]^{\pi_A} \\
        \Gamma.A \ar[r]_-{\pi_A} & \Gamma
    }
\end{equation*}
is a pullback; so by assumption, we have a factorisation of the diagonal
morphism $\delta_A \colon \Gamma.A \to \Gamma.A.A^+$ as
\begin{equation*}
 \Gamma.A \xrightarrow{i_{\pi_A}} I(\pi_A) \xrightarrow{j_{\pi_A}} \Gamma.A.A^+
\end{equation*}
where $i_{\pi_A}$ is a cloven $\ELL$-map and $j_{\pi_A}$ a cloven $\R$-map. We
now take the identity type $\Id_A \in \Ty(\Gamma.A.A^+)$ to be the cloven
$\R$-map $j_{\pi_A}$, and take the introduction morphism $r_A \colon \Gamma.A
\to \Gamma.A.A^+.\Id_A$ to be $i_{\pi_A}$. Note that $\pi_{\Id_A}.r_A =
j_{\pi_A}.i_{\pi_A} = \delta_A$ as required. As for the elimination and
computation rules, let us suppose given $C \in \Ty(\Gamma.A.A^+.\Id_A)$ and a
commutative square of the form~\eqref{idelim}. In such a square, the dependent
projection $\pi_C$ comes with an assigned cloven  $\R$-map structure, whilst
$r_A$ bears a cloven $\ELL$-map structure (since it is $i_{\pi_A}$); and so by
Proposition~\ref{howtogetfillers}, we obtain the required choice of diagonal
filler $J(C,d) \colon \Gamma.A.A^+.\Id_A \to \Gamma.A.A^+.\Id_A.C$.

It remains to verify the stability of the above structure under substitution.
Firstly, given $A \in \Ty(\Gamma)$ and $f \colon \Delta \to \Gamma$ in $\E$, we
must show that $\Id_A[f^{++}] = \Id_{A[f]}$. But it follows from stability of
the choice of diagonal factorisations that there is a pullback square
\begin{equation*}
    \cd[@+0.5em]{
        \Delta.A[f].A[f]^+.\Id_{A[f]} \ar[r]^-{\theta} \ar[d]_{\pi_{\Id_{A[f]}}} &
        \Gamma.A.A^+.\Id_A \ar[d]^{\pi_{\Id_A}} \\
        \Delta.A[f].A[f]^+ \ar[r]_-{f^{++}} & \Gamma.A.A^+
    }
\end{equation*}
in $\E$; and so it suffices to show that in this square, pulling back the
assigned cloven $\R$-map structure on the right-hand
arrow yields the assigned $\R$-map structure on the left-hand one. For this we
observe that---by functoriality of the choice of diagonal factorisations---when
both vertical maps are equipped with their assigned $\R$-map structures, the
displayed square comprises a morphism of $\R$-maps. But by the remark
concluding the proof of Proposition~\ref{typecatconstruct}, this condition
characterises uniquely the $\R$-map structure induced by pullback on the
left-hand arrow; so that the pullback $\R$-map structure coincides with the
assigned one, as required.
%
%
%

Finally, we verify commutativity of the diagrams~\eqref{subststabdiags}
and~\eqref{subststabdiags2}. For the former, this follows immediately from
functoriality of the choice of diagonal factorisations; whilst for the latter,
we observe that---again by functoriality---the vertical maps
in~\eqref{subststabdiags} comprise a morphism of cloven $\ELL$-maps from
$r_{A[f]}$ to $r_A$; and that in the pullback diagram
\begin{equation*} \cd[@C+3em]{
    \llap{$\Delta.A[f].A[f]^+.$}\Id_{A[f]}.C[f^{+++}] \ar[r]^-{f^{++++}} \ar[d]_{\pi_{C[f^{+++}]}} & \Gamma.A.A^+.\Id_{A}\rlap{$.C$} \ar[d]^{\pi_C}\\
    \llap{$\Delta.A[f].$}A[f]^+.\Id_{A[f]} \ar[r]_-{f^{+++}} & \Gamma.A.A^+.\Id_{A}\ \text,
  }
\end{equation*}
the horizontal arrows comprise a morphism of cloven $\R$-maps from
$\pi_{C[f^{+++}]}$ to $\pi_C$. Commutativity in~\eqref{subststabdiags2}
therefore follows from Proposition~\ref{naturalfillers}.
\end{proof}
This does not quite complete the proof of Theorem~\ref{mainthm1}, since a
categorical model of identity types must model the \emph{strong} elimination
and computation rules, whereas the previous Proposition only indicates how to
model the standard ones. For this we make use the Frobenius property of a
homotopy-theoretic model:
%
%
\begin{proof}[of Theorem~\ref{mainthm1}]
By Proposition~\ref{typecathasids}, we know that we already have a type
category with identity types; hence it suffices to show that these identity
types are strong. So suppose given a commutative diagram as
in~\eqref{strongiddiag}. By assumption, $\pi_C$ is a cloven $\R$-map; and by
repeated application of the Frobenius structure, $(r_A)^{+ \cdots +}$ is a
cloven $\ELL$-map; and so we may once again apply
Proposition~\ref{howtogetfillers} to obtain the required diagonal filler. All
that remains is to verify the stability of these new fillers under
substitution; and this follows exactly as before, but now using also the
functoriality of the Frobenius structure.
\end{proof}

\section{Path object categories}\label{catframework}
We have now described the notion of a homotopy-theoretic model of identity
types, and shown that any such gives rise to a categorical model. However, as
we noted in Remark~\ref{nontrivial}, it is in practice rather hard to verify
that a category satisfies the axioms for a homotopy-theoretic model. We are
therefore going to introduce a further set of axioms which are much simpler to
verify, but still sufficiently strong to allow the construction of a
homotopy-theoretic model. We call a category satisfying these axioms a
\emph{path object category}. This section gives the definition; the following
one provides a number of examples; and Section~\ref{typecatstruct} proves our
second main result, that every path object category gives rise to a
homotopy-theoretic model of identity types.


\subsection{Path objects} Throughout the rest of this section, we assume given
a finitely complete category $\E$. The first piece of structure that we will
require is a choice of ``path object'' for each object of $\E$.

\vskip0.5\baselineskip\textsc{Axiom 1}.\label{pathobject} For every $X \in \E$,
there is given an internal category
\begin{equation*}
    \cd[@C+1em]{
        X \ar[r]|-{r_X} & MX \ar@<-6pt>[l]_-{s_X} \ar@<6pt>[l]^-{t_X} & MX
\,{}_{s_X}\!\!\times_{t_X} MX \ar[l]_-{{}\,\,\,m_X}
    }\ \text.
\end{equation*}
together with an involution $\tau_X \colon MX \to MX$ which defines an internal
identity-on-objects isomorphism between the category $MX$ and its opposite. We
require this structure to be functorial in $X$---thus the assignation $X
\mapsto MX$ should underlie a functor $\E \to \E$, and the maps $s_X$, $t_X$,
$m_X$, $r_X$ and $\tau_X$ should be natural in $X$---and that the functor $M$
should preserve pullbacks.
\vskip0.5\baselineskip

\begin{Rk}
We may rephrase the above in a number of ways. The naturality of $r$, $m$, $s$
and $t$ is equivalent to asking that every morphism $f \colon X \to Y$ of $\E$
should induce an internal functor $(f, Mf) \colon (X, MX) \to (Y, MY)$; and
this is in turn equivalent to asking that we should have a functor $\E \to
\cat{Cat}(\E)$ which is a section of the functor $\mathrm{ob} \colon
\cat{Cat}(\E) \to \E$.
\end{Rk}

The example that we have in mind---which will be developed in more detail in
the following section---is the category of topological spaces. Here the first
thing one might try is to take $MX \defeq X^{[0,1]}$, with $r_X$ and $m_X$
given by the constant path and concatenation of paths respectively. However,
this definition fails to satisfy the category axioms, since concatenation of
paths is only associative and unital ``up to homotopy''. To rectify this, we
take $MX$ instead to be the \emph{Moore path space} of $X$, given by the set
$    \set{(r, \phi)}{r \in \mathbb R^+\text,\ \phi \colon [0, r] \to X}$
equipped with a suitable topology (detailed in Section~\ref{exs} below). Now
$r_X$ sends a point of $X$ to the path of length $0$ at that point, whilst
$m_X$ takes paths of length $r$ and $s$ respectively to the concatenated path
of length $r+s$. With this definition, we determine an internal category $MX
\rightrightarrows X$ as required.


\subsection{Strength} \label{strengthsection} Returning to our axiomatic framework, the next piece of structure we
wish to encode is a way of contracting a path onto one of its endpoints. As a
first attempt at formalising this, we might require that for every path $\phi
\colon a \to b$ in $X$ there be given a ``path of paths'' $\eta_X(\phi)$ of the
form indicated by the following diagram:
\begin{equation*}
    \cd[@!@+0.5em]{
    a \ar[r]^\phi \ar[d]_\phi \rtwocell{dr}{\eta_X(\phi)} &
    b \ar[d]^{r_X(b)} \\
    b \ar[r]_{r_X(b)} & b
    } \ \text.
\end{equation*}
This amounts to asking for a morphism $\eta_X \colon MX \to MMX$ which yields
the identity upon postcomposition with $s_{MX}$ or $Ms_X$, and yields the
composite $r_X . t_X$ upon postcomposition with $t_{MX}$ or $Mt_X$. However,
this formalisation fails to capture our motivating topological example. Indeed,
for a Moore path $\phi \colon a \to b\in MX$ of length $r$, the corresponding
path of paths $\eta_X(\phi)$ must also be of length $r$, since $Ms_X$ preserves
path length and we require that $Ms_X(\eta_X(\phi)) = \phi$. But then whenever
$r > 0$ we cannot have $Mt_X(\eta_X(\phi)) = r_X(b)$, since $Mt_X$ also
preserves path length and $r_X(b)$ is of length $0$.

We must therefore refine our formalisation. In the motivating topological case,
we can do this by requiring that $Mt_X(\eta_X(\phi))$ be not a path of length
0, but rather a path of the same length as $\phi$, but constant at $b$. In
order to capture the idea of a non-identity, but constant, path in our abstract
framework, we introduce a further piece of structure.

\vskip0.5\baselineskip\textsc{Axiom 2}. The endofunctor $M$ comes equipped with
a \emph{strength}~\cite{Kock1970Monads}
\begin{equation*}
\alpha_{X, Y} \colon MX \times Y \to M(X \times Y)
\end{equation*}
with respect to which $s$, $t$, $r$, $m$ and $\tau$ are strong natural
transformations.\vskip0.5\baselineskip

In fact, such a strength is determined, up to natural isomorphism, by its
components of the form $\alpha_{1, X} \colon M1 \times X \to MX$, since every
square of the following form is a pullback:
\begin{equation}\label{pbcomponents}
    \cd[@C+0.5em]{
      MX \times Y \ar[r]^-{\alpha_{X,Y}} \ar[d]_{M! \times Y} &
      M(X \times Y) \ar[d]^{M\pi_2} \\
      M1 \times Y \ar[r]_-{\alpha_{1,Y}} &
      MY\rlap{ .}
    }
\end{equation}
To see this, we consider the diagram
\begin{equation*}
    \cd[@C+0.5em]{
      MX \times Y \ar[r]^-{\alpha_{X,Y}} \ar[d]_{M! \times Y} &
      M(X \times Y) \ar[r]^-{M\pi_1} \ar[d]_{M\pi_2} &
      MX \ar[d]^{M!} \\
      M1 \times Y \ar[r]_-{\alpha_{1,Y}} &
      MY \ar[r]_-{M!} &
      M1\rlap{ .}
    }
\end{equation*}
In it, the right-hand square is a pullback since $M$ preserves pullbacks, and
the outer square is a pullback since the composites along the top and bottom
are equal, by naturality of $\alpha$, to $\pi_1$. Hence the left-hand square is
also a pullback as claimed.

We may give the following interpretation of the components $\alpha_{1,X}$. The
object $M1$ we think of as the ``object of path lengths'', and the morphism
$\alpha_{1, X} \colon M1 \times X \to MX$ as taking a length $r \in M1$ and an
element $x \in X$ and returning the path of length $r$ which is constant at
$x$. In fact, by naturality of $\alpha$ in $Y$ together with strength of $t$,
we see that the diagram
\begin{equation}\label{retractdiagram}
    M1 \times X \xrightarrow{\alpha_{1, X}} MX \xrightarrow{(M!,\, t_X)} M1
\times X
\end{equation}
exhibits $M1 \times X$ as a retract of $MX$, so that we may think of $M1 \times
X$ as being the ``subobject of constant paths'' of $MX$. In our topological
example, we have a strength $\alpha_{X, Y}$ that takes a Moore path $(r, \phi)
\in MX$ and a point $y \in Y$ and returns the path $(r, \psi) \in M(X \times
Y)$, where $\psi(s) = (\phi(s), y)$. Now the subspace of $MX$ determined by the
retract~\eqref{retractdiagram} is comprised precisely of the constant paths of
arbitrary length in $X$, as desired.

\subsection{Path contraction} With the aid of Axiom 2, we may now formalise the structure allowing the
contraction of a path onto its endpoint. Observe that it is in
equation~\eqref{Mt} that we make use of the notion of constant path.

\vskip0.5\baselineskip\textsc{Axiom 3}. There is given a strong natural
transformation $\eta \colon M \Rightarrow MM$ such that the following equations
hold:
\begin{align}
    s_{MX} .\eta_X & = \id_{MX} \label{seta}\\
    t_{MX} .\eta_X & = r_X. t_X \label{teta}\\
    Ms_{X} .\eta_X & = \id_{MX} \label{Ms}\\
    Mt_{X} .\eta_X & = \alpha_{1,X} . (M!, t_X) \label{Mt}\\
    \eta_{X} . r_X & = r_{MX} . r_X \ \text.\label{tri2}
\end{align}
\vskip0.5\baselineskip
\begin{Defn}
By a \emph{path object category}, we mean a finitely complete category~$\E$
satisfying Axioms~1, 2 and 3.
\end{Defn}
Our motivation for introducing the notion of path object category is the
following theorem, whose proof we give in Section \ref{typecatstruct} below.
\begin{Thm}\label{secondmain}
A path object category is a homotopy-theoretic model of identity types; and
hence also a categorical model of identity types.
\end{Thm}
\begin{Rk}\label{allowableweakening}
The above structure is in fact slightly more than is necessary for our
argument. Axiom 1 posits an internal category structure $MX \rightrightarrows
X$ on each object $X$ of $\E$, but we will make no use at all of the
associativity of composition in these internal categories, and need right
unitality only at one particular (though crucial) point; on the other hand, we
will make repeated and unavoidable use of left unitality. Now, for the examples
we have in mind, it happens that, in ensuring the unitality axioms, we also
obtain associativity: but in recognition of other potential examples where this
may not be so, we give a more precise delineation of what is needed. On
replacing Axiom 1 by the following weaker Axiom 1$'$, we may still carry
through the entire proof of Theorem~\ref{secondmain}, as given in
Section~\ref{typecatstruct} below; and on replacing it by the yet weaker Axiom
1$''$, may still carry out at least the arguments of
Sections~\ref{sec:clovenwfs} and~\ref{sec:functorialfrob}, though that of
Section~\ref{facdiag} does not go through, as without right unitality we cannot
complete the proof of Proposition~\ref{choiceofdiagonalfacs}.

\vskip0.5\baselineskip\textsc{Axiom 1$'$}. As Axiom 1, but drop associativity
of composition, and functoriality of $\tau$ with respect to binary composition.
Explicitly, this means that for every $X \in \E$, there is given, functorially
in $X$, a diagram
\begin{equation*}
    \cd[@C+1em]{
        X \ar[r]|-{r_X} & MX \ar@(ul,ur)[]|{\tau_X} \ar@<-6pt>[l]_-{s_X} \ar@<6pt>[l]^-{t_X} & MX
\,{}_{s_X}\!\!\times_{t_X} MX \ar[l]_-{{}\,\,\,m_X}
    }\ \text,
\end{equation*}
such that $M$ preserves pullbacks, and such that the following equations hold:
\begin{align*}
    s_{X} . r_X & = 1_{X} & m_{X}  (r_X.t_X,\, 1_{MX}) & = 1_{MX}\\
    t_{X} . r_X & = 1_{X} &s_{X} . \tau_X & = t_{X} \\
    s_{X} . m_X & = s_X. \pi_2 & t_{X} . \tau_X & = s_{X} \\
    t_{X} . m_X & = t_X. \pi_1 &    r_X  & = \tau_X . r_X \\
    m_{X}  (1_{MX},\, r_X.s_X) & = 1_{MX} &     \tau_{X} . \tau_X & = 1_X\ \text.
\end{align*}
\vskip0.5\baselineskip\textsc{Axiom 1$''$}. As Axiom 1$'$, but drop also the
right unitality axiom (uppermost in the right-hand column of the preceding
list).\vskip0.5\baselineskip

In spite of the above, we shall continue to write as if the stronger Axiom 1
were satisfied, for two reasons. Firstly, this allows us all the terminological
conveniences afforded by the language of internal category theory; and
secondly, as noted above, in each of the examples we shall now give it
\emph{is} the stronger Axiom 1 that holds.
\end{Rk}



\section{Examples of path object categories}\label{exs}
\subsection{Topological spaces}
\begin{Prop}
The category of topological spaces may be equipped with the structure of a path
object category.
\end{Prop}
\begin{proof}
As discussed above, we cannot define the path space $MX$ of a topological space
$X$ to be $X^{[0,1]}$, since composition of paths $[0,1] \to X$ is associative
and unital only ``up to homotopy''. Instead, we consider paths of varying
lengths and the \emph{Moore path space}
\[ MX = \{ (\ell, \phi) \, | \, \ell \in \mathbb{R}_{+},\, \phi \colon [0, \ell] \to X \} \ \text.\]
(Here ${\mathbb R}_+$ denotes the set of non-negative real numbers.) To define
a suitable topology on $MX$, we observe that there is a monomorphism $MX
\hookrightarrow \mathbb{R}_{+} \times X^{\mathbb{R}_{+}}$  sending $(\ell,
\phi)$ to the pair $(\ell, \overline{\phi})$, where $\overline{\phi}$ is
defined by
\begin{eqnarray*}
\overline{\phi}(x) & = & \left\{ \begin{array}{ll}
\phi(x) & \mbox{if } x \leq \ell\text, \\
\phi(\ell) & \mbox{if } x \geq \ell\text.
\end{array} \right.
\end{eqnarray*}
We may therefore topologise $MX$ as a subspace of $\mathbb{R}_{+} \times
X^{\mathbb{R}_{+}}$, where the latter is equipped with its natural topology
(coming from the product and the compact-open topology). We have continuous
projections $s_X,t_X \colon MX \rightrightarrows X$ sending $(\ell, \phi)$ to
$\phi(0)$ and $\phi(\ell)$ respectively; and as is well-known, the topological
graph this defines  bears the structure of a topological category, with the
identity $r_X(x)$ being the constant path at $x$ of length $0$, and the
composition $m_X$ of paths $(\ell, \phi)$ and $(m, \psi)$ being the path $\chi$
of length $\ell + m$ given by:
\begin{eqnarray*}
\chi(x) & = & \left\{ \begin{array}{ll}
\phi(x) & \mbox{if } x \leq \ell, \\
\psi(x - \ell) & \mbox{if } x \geq \ell.
\end{array} \right.
\end{eqnarray*}
Moreover, this topological category bears an involution $\tau_X$, which sends a
path $\phi$ of length $\ell$ to the same path in reverse: that is, the path
$\phi^o$ of length $\ell$ with $\phi^o(x) = \phi(\ell-x)$.

We now define a strength $\alpha_{X, Y} \colon MX \times Y \to M(X \times Y)$
by sending the pair $((\ell, \phi), y)$ to the path $\psi$ of length $\ell$ in
$X \times Y$  defined by $\psi(x) = (\phi(x), y)$. Finally, the maps $\eta_X
\colon MX \to MMX$ are given by contracting a path to its endpoint: so we take
$\eta_X(\ell, \phi)$ to be the path $\psi$ in $MX$ of length $\ell$ whose value
at $x \leq \ell$ is the path of length $\ell - x$ given by $\psi(x)(y) =
\phi(x+y)$.
We have now given all the data required for a path object category; the
remaining verifications are entirely straightforward and left to the reader.
\end{proof}
\begin{Cor}
The category of topological spaces may be equipped with the structure of a
categorical model of identity types. In this model, propositional equalities
between terms of closed type are interpreted by Moore homotopies.
\end{Cor}

\subsection{Groupoids}
In our next few examples, we revisit some previously constructed models of
identity types and show that they fit into our framework, beginning with one of
the earliest instances of a higher-dimensional model of type theory: the
\emph{groupoid model} of~\cite{Hofmann1998groupoid}.
\begin{Prop}\label{gpdmodel}
The category of small groupoids bears a structure of path object category.
\end{Prop}
\begin{proof}
Let $\I$ be the groupoid with two objects $0$ and $1$ and two non-identity
arrows $0 \to 1$ and $1 \to 0$ (which clearly must be each others' inverses).
Now for any groupoid $X$, we define $MX = X^\I$; so the objects of $MX$ are the
arrows in $X$ and the morphisms of $MX$ are commutative squares. The functors
$s_X, t_X \colon MX \to X$ are obtained by taking domains and codomains,
respectively, whilst the functor $r_X \colon X \to MX$ sends an object of $X$
to the identity arrow at that object. The functor $m_X\colon MX \times_X MX \to
MX$ sends a pair of arrows $(g,f)$ to their composite $gf$; whilst $\tau_X
\colon MX \to MX$ sends each arrow to its inverse. It is immediate that these
data equip $X$ with the structure of an internal groupoid. We obtain the
strength $\alpha_{X, Y} \colon MX \times Y \to M(X \times Y)$ by mapping an
arrow $f$ of $X$ and an object $y$ of $Y$ to the arrow $(f, \id_y)$ of $X
\times Y$; and finally, we define the maps $\eta_X \colon MX \to MMX$ by
sending an arrow $f \colon a \to b$ of $X$ to the commuting square
\begin{equation*}
\cd[@-0.5em]{ a \ar[r]^f \ar[d]_f & b \ar[d]^\id \\ b \ar[r]_\id & b\rlap{ .}}
\end{equation*}
Again, the remaining verifications are straightforward.
\end{proof}
\begin{Cor}
The category of small groupoids may be equipped with the structure of a
categorical model of identity types. In this model, propositional equalities
between terms of closed type are interpreted by natural isomorphisms between
functors.
\end{Cor}

\subsection{Chain complexes}
For our next example we consider chain complexes over a ring. The model of type
theory this induces is also constructed in~\cite{Warren2008Homotopy}.
\begin{Defn}
Let $R$ be a ring. A \emph{chain complex} over $R$ is given by a collection $A
= (A_n \, | \, n \in \mathbb Z)$ of left $R$-modules, together with maps
$\partial_n \colon A_n \to A_{n-1}$ satisfying $\partial_n \partial_{n+1} = 0$
for all $n \in \mathbb Z$. Given chain complexes $A$ and $B$, a \emph{chain
map} $A \to B$ is a family of maps $(f_n \colon A_n \to B_n \,  | \, n \in
\mathbb Z)$ such that $\partial_n \circ f_n = f_{n-1} \circ
\partial_n$ for all $n \in \mathbb Z$.
\end{Defn}
\begin{Prop}
The category of chain complexes over a ring $R$ bears the structure of a path
object category.
\end{Prop}
\begin{proof}
For a chain complex $A$, we define the path object $MA$ by
\begin{eqnarray*}
(MA)_n & = & A_n \oplus A_{n+1} \oplus A_n , \\
\partial_n(a, f, b) & = & (\partial_n a,\, b - a - \partial_{n+1} f,\, \partial_n b).
\end{eqnarray*}
The source and target maps $s_A, t_A \colon MA \to A$ are given by first and
third projection respectively, whilst $r_A \colon A \to MA$ is given by $r_A(a)
= (a, 0, a)$. In order to define the composition map $m_A \colon MA \times_A MA
\to MA$, note that we have $(MA \times_A MA)_n \cong A_n \oplus A_{n+1} \oplus
A_n \oplus A_{n+1} \oplus A_n$, so that we may take $m_A(a, f, b, g, c) = (a, f
+ g, c)$. Finally, the involution $\tau_A \colon MA \to MA$ is given by
$\tau_A(a, f, b) = (b, -f, a)$. It is easy to see that all the maps defined so
far are chain maps and give $A$ the structure of an internal groupoid. Note
that $M$ preserves pullbacks, as it does so degreewise; in fact, it also
preserves the terminal object $0$, and hence all finite limits. We give a
strength $\alpha_{A,X} \colon MA \oplus X \to M(A \oplus X)$ by $
\alpha_{A,X}\big((a, f, a'), x\big) = \big((a, x), (f,0), (a', x)\big)$.
Finally, to give the maps $\eta_A \colon MA \rightarrow MMA$, we observe
that since
\[ (MMA)_n  =  (A_n \oplus A_{n+1} \oplus A_n) \oplus (A_{n+1} \oplus A_{n+2} \oplus A_{n+1}) \oplus  (A_n \oplus A_{n+1} \oplus A_n)\ \text, \]
we may take $\eta_A(a, f, b) = \big((a, f, b), (f, 0, 0), (b, 0, b)\big)$. The
remaining verifications are entirely straightforward.
\end{proof}
Recall that if $f,g \colon A \to B$ are chain maps, then a \emph{chain
homotopy} $\gamma \colon f \Rightarrow g$ is a collection of $R$-module maps
$\gamma_n \colon A_n \to B_{n+1}$ satisfying $\partial_{n+1} \gamma_n +
\gamma_{n-1} \partial_n = g_n - f_n$ for all $n \in \mathbb Z$. Observing that
chain maps $c \colon A \to MB$ are in bijection with chain homotopies $s_A.c
\Rightarrow t_A.c$, we conclude that:
\begin{Cor}
The category of chain complexes over a ring $R$ may be equipped with the
structure of a categorical model of identity types. In this model propositional
equalities between terms of closed type are interpreted by chain homotopies.
\end{Cor}

\subsection{Interval object categories}
The preceding two models of identity types were shown
in~\cite{Warren2008Homotopy} to be part of a class of such models that may be
constructed from categories equipped with an \emph{interval object}. In fact,
every such category gives us an example of a path object category, so that the
machinery he describes can be understood as a special case of that developed
here. Let us first recall the salient definitions
from~\cite{Warren2008Homotopy,Warren2009characterization}.

\begin{Defn}
Let $(\E, \otimes, I)$ be a symmetric monoidal closed category.
\begin{itemize}
\item A (strict, invertible) \emph{interval} is a cogroupoid object
\begin{equation*}
    \cd[@C+1em]{
        I \ar@<-6pt>[r]_-{\top} \ar@<6pt>[r]^-{\bot}  & A \ar[l]|-{i} \ar[r]_-{\star} & A +_I A
    }
\end{equation*}
 whose ``object of co-objects'' is the unit of the monoidal
structure.\vskip0.5\baselineskip
\item A \emph{join} on an interval is a morphism $\mathord \vee \colon A
    \otimes A \to A$ making the following diagrams commute:
\begin{equation*}
  \cd{
    A \ar[r]^-{A \otimes \bot} \ar[dr]_{1_A} &
    A \otimes A \ar[d]^{\vee} &
    A \ar[l]_-{\bot \otimes A} \ar[dl]^{1_A} \\ &
    A
  } \qquad
  \cd{
    A \ar[r]^-{A \otimes \top} \ar[d]_{i} &
    A \otimes A \ar[d]^{\vee} &
    A \ar[l]_-{\top \otimes A} \ar[d]^{i} \\
    I \ar[r]_{\top} &
    A &
    I \ar[l]^{\top}
  } \qquad
  \cd{
    A \otimes A \ar[r]^-{\vee} \ar[dr]_{i \otimes i} &
    A \ar[d]^{i} \\ &
    I
  }
\end{equation*}
\item An \emph{interval object category} is a finitely complete symmetric
    monoidal category $\E$ equipped with a strict, invertible interval with
    a join.
\end{itemize}
\end{Defn}
\begin{Ex}
Both the category of small groupoids and the category of chain complexes may be
made into interval object categories. The interval in the category of groupoids
is $1 \rightrightarrows \I$, where $\I$ is the free-living isomorphism
described in Proposition~\ref{gpdmodel} above. In the category of chain
complexes over a ring $R$, the interval is $I \rightrightarrows A$, where the
$I$ is the chain complex which is $R$ in degree zero and is $0$ elsewhere; and
$A$ is the chain complex which is $R$ in degree one, $R \oplus R$ in degree
zero, and $0$ elsewhere;
and whose only non-trivial differential $\delta_1 \colon R \oplus R \to R$ is
given by $\delta_1(x) = (x,-x)$. See~\cite[Examples
1.3]{Warren2009characterization} for the further details.
\end{Ex}
\begin{Prop}
Any interval object category is a path object category.
\end{Prop}
\begin{proof}
We define $MX \defeq [A, X]$; and since the functor $[\thg, X] \colon \E^\op
\to \E$ sends colimits to limits, the cogroupoid structure on $I
\rightrightarrows A$ transports to a groupoid structure on $MX = [A,X]
\rightrightarrows [I, X] \cong X$, naturally in $X$. Observe that since $M$ has
a left adjoint, it preserves all limits, and so certainly all pullbacks. The
strength on $M$ is defined by the composite
\begin{equation*}
    MX \times Y \xrightarrow{1 \times r_Y} MX \times MY \xrightarrow{\cong} M(X \times Y)\ \text,
\end{equation*}
where the unnamed isomorphism expresses the product-preservation of $M$.
Finally, the maps $\eta_X \colon MX \to MMX$ are given by
\begin{equation*}
    MX = [A,X] \xrightarrow{[\mathord \vee, X]} [A \otimes A, X] \cong [A, [A, X]] = MMX\ \text.
\end{equation*}
The remaining axioms for a path object category now follow directly from those
for an interval object category.
\end{proof}
\begin{Cor}
Any interval object category may be equipped with the structure of a
categorical model of identity types.
\end{Cor}

\subsection{Simplicial sets}\label{ssetssection}
Our final example of a path object category is the category of simplicial sets.
Exhibiting the necessary structure will be somewhat more involved than for the
preceding examples; on which account we only sketch the construction here,
reserving the detailed combinatorics for Section~\ref{simplicialexample} below.
We begin with some basic definitions.

\begin{Defn}
Let $\Delta$ be the category whose objects are the ordered sets $[n] = \{0,
\dots, n\}$ (for $n \in \mathbb N$) and whose morphisms are order-preserving
maps. A \emph{simplicial set} is a functor $\Delta^\op \to \cat{Set}$. If $X$
is a simplicial set, we write $X_n$ for $X([n])$, and call its elements the
\emph{$n$-simplices} of $X$. The category of simplicial sets, $\cat{SSet}$, is
the presheaf category $[\Delta^\op, \cat{Set}]$.
\end{Defn}
We think of the $0$-, $1$-, $2$- and $3$-simplices of $X$ as being points,
lines, triangles and tetrahedra respectively; more generally, we visualise a
typical simplex $x \in X_n$ as the convex hull of $n+1$ (ordered) points in
general position in $\mathbb R^n$, whose interior is labelled with $x$ and
whose faces are labelled by the simplices obtained by acting on $x$ by
monomorphisms $[k] \to [n]$ of $\Delta$.

\begin{Prop}\label{ssets}
The category of simplicial sets bears the structure of a path object category.
\end{Prop}
Observing that the category of simplicial sets is cartesian closed, an obvious
candidate for the path object of a simplicial set $X$ would be the internal hom
$X^{\Delta^1}$, where $\Delta^1 \defeq \Delta(\thg,1)$ is the free simplicial
set on a $1$-simplex. However, much as in the topological case, this most
obvious candidate fails to satisfy the axioms required of it; though this time
in a more serious manner. Indeed, for an arbitrary simplicial set $X$, it need
not be the case that two paths $\Delta^1 \to X$ with matching endpoints even
have a composite: the generic counterexample being given by the simplicial set
$\Lambda^2_1$ freely generated by the diagram
\begin{equation}\label{kan1}
\cd[@-1em]{ & 1 \ar[dr]^g \\ 0 \ar[ur]^f & & 2\rlap{ ,}}
\end{equation}
in which the paths $f, g \colon \Delta^1 \to \Lambda^2_1$ are evidently not
composable. We might try and overcome this by  restricting attention to those
simplicial sets $X$ which are \emph{Kan complexes}: a condition which requires,
amongst other things, that any diagram of the shape~\eqref{kan1} in $X$ should
have an extension to a $2$-simplex. But though this would allow us to compose
paths $\Delta^1 \to X$---and more generally, to construct a composition
operation $m_X \colon X^{\Delta^1} \times_X X^{\Delta^1} \to X^{\Delta^1}$---we
would then encounter the same problem as we did in the topological case:
namely, that this composition would only be associative and unital ``up to
higher simplices'', rather than on the nose. For this reason, we shall not
consider Kan complexes any further; instead, taking our inspiration from the
topological case, we intend to define a simplicial analogue of the Moore path
space, which will allow us to equip \emph{any} simplicial set with the
structure of an internal category. Our motivating definition is the following:
\begin{Defn}\label{moore0}
Let $X$ be a simplicial set and let $x, x'$ be $0$-simplices in $X$. A
\emph{simplicial Moore path} from $x$ to $x'$ is given by $0$-simplices $x =
z_0, z_1, \dots, z_k = x'$ and $1$-simplices $f_1, \dots, f_k$ such that for
each $1 \leqslant i \leqslant k$, either $f_i \colon z_{i-1} \to z_i$ or $f_i
\colon z_i \to z_{i-1}$.
\end{Defn}
A typical such Moore path looks like:
\begin{equation}\label{typ0}
\cd{ x \ar[r]^{f_1} & z_1 \ar[r]^{f_2} & z_2 & z_3 \ar[l]_{f_3} \ar[r]^{f_4} & z_4 & x' \ar[l]_{f_5}}
\end{equation}
and our intention is that the totality of such Moore paths in $X$ should
provide the $0$-simplices of the path object $MX$. Note that these Moore paths
are composable---by placing them end to end---and reversible, and that the
composition is associative and unital (with identities given by the empty
path); so that, at the $0$-dimensional level at least, we have the data
required for Axiom~1 of a path object category.

With regard to the higher-dimensional simplices of $MX$, we reserve a formal
description for Definition~\ref{mooren} below: but may at least provide the
following informal picture of what they are. Given $n$-simplices $\xi, \xi' \in
X$, a simplicial Moore path between them will consist of $n$-simplices $\xi =
\zeta_0$, \dots, $\zeta_k = \xi'$ together with $(n+1)$-simplices $\phi_1$,
\dots, $\phi_k$ such that for each $1 \leqslant i \leqslant k$, the simplex
$\phi_i$ mediates between $\zeta_i$ and $\zeta_{i+1}$ in a suitably disciplined
manner. We do not wish to make this precise here, but instead give the
following example of a typical $1$-dimensional Moore path by way of
illustration:
%
%
\begin{equation}\label{typ1}
\cd[@+1em]{
  x \ar[r]^{f_1} \ar[d]_\xi \ar[dr]|{\zeta_1} \ar[drr]|{\zeta_2} &
  z_1 \ar[r]^{f_2} \ar[dr]|{\zeta_3} &
  z_2 \ar[d]|{\zeta_4} &
  z_3 \ar[l]_{f_3} \ar[dl]|{\zeta_5} \ar[d]|{\zeta_6} \ar[r]^{f_4} \ar@{}[dr]|(0.3){\phi_7} &
  z_4 \ar[dl]|{\zeta_7} \ar@{}[d]|(0.3){\phi_8} &
  y\rlap{ .} \ar[l]_{f_5} \ar[dll]^{\xi'} \\
  x' \ar[r]_{f'_1} \ar@{}[ur]|(0.3){\phi_1}&
  z'_1 \ar@{}[u]|(0.25){\phi_2} \ar@{}[u]|(0.75){\phi_3} \ar@{}[ur]|(0.7){\phi_4} &
  z'_2 \ar[l]^{f'_2} \ar[r]_{f'_3} &
  y' \ar@{}[ul]|(0.7){\phi_5} \ar@{}[ul]|(0.3){\phi_6} & {}
}
\end{equation}
We take the set of $n$-simplices of $MX$ to be the $n$-dimensional Moore paths
in $X$; and upon doing so, see that---as in the $0$-dimensional case---such
Moore paths are composable (again by placing them end to end) and reversible,
with the composition being associative and unital. So in this way we obtain the
internal category with involution $MX \rightrightarrows X$ required for Axiom
1. For Axiom 2, we are required to give maps $\alpha_{X,Y} \colon MX \times Y
\to M(X \times Y)$, and we illustrate the construction through examples at the
$0$-\ and $1$-dimensional level. For a Moore path~\eqref{typ0} in $X$ and a
$0$-simplex $y \in Y$, the corresponding path in $X \times Y$ will be
\begin{equation*}
\cd{ (x,y) \ar[r]^{(f_1,1_y)} & (z_1,y) \ar[r]^{(f_2,1_y)} & (z_2,y) & (z_3,y) \ar[l]_{(f_3,1_y)} \ar[r]^{(f_4,1_y)} & (z_4,y) & (x',y) \ar[l]_{(f_5,1_y)}}
\end{equation*}
where the $1$-simplex $1_y \colon y \to y$ is the image of $y$ under the unique
epimorphism $[1] \to [0]$ in $\Delta$. Similarly, given a $1$-dimensional Moore
path like~\eqref{typ1} in $X$ and a $1$-simplex $h \colon y \to y'$ in $Y$, the
corresponding Moore path in $X \times Y$ will have~\eqref{typ1} as its
projection on to $X$; and as its projection on to $Y$, the diagram
\begin{equation*}
\cd[@+1em]{ y \ar[r]^{1_y} \ar[d]_h \ar[dr]|h \ar[drr]|h & y \ar[r]^{1_y} \ar[dr]|h & y \ar[d]|h & y \ar[l]_{1_y} \ar[dl]|h \ar[d]|h \ar[r]^{1_y} & y \ar[dl]|h & y\rlap{ .} \ar[l]_{1_y} \ar[dll]^{h} \\ y' \ar[r]_{1_{y'}} & y' & y' \ar[l]^{1_{y'}} \ar[r]_{1_{y'}} & y'}
\end{equation*}
in which all of the interior $2$-simplices are obtained by acting on $h$ by
suitable epimorphisms $[2] \to [1]$ of $\Delta$.

Finally, we consider Axiom 3, for which we must find maps $\eta_X \colon MX \to
MMX$ contracting a path on to its endpoint. For the purposes of this section,
we illustrate this only with an example at the $0$-dimensional level. Given a
Moore path such as~\eqref{typ1}, we are required to produce a $0$-simplex of
$MMX$: that is, a $0$-dimensional Moore path in $MX$, which, if we draw it down
the page, will be given as follows:
\begin{equation}\label{contractsample}
\cd[@+1em]{
  x \ar[r]^{f_1} \ar[d]_{f_1} &
  z_1 \ar[r]^{f_2} \ar[dl]|{1_{z_1}} \ar[d]|{f_2} &
  z_2 \ar[dl]|{1_{z_2}} &
  z_3 \ar[l]_{f_3} \ar[r]^{f_4} \ar[dll]|{f_3} \ar[dl]|{1_{z_3}} \ar[d]|{f_4} &
  z_4 \ar[dl]|{1_{z_4}} &
  x' \rlap{ .} \ar[l]_{f_5} \ar[dll]|{f_5} \ar[dl]^{1_{x'}} \\
  z_1 \ar[r]|{f_2} \ar[d]_{f_2} &
  z_2 \ar[dl]|{1_{z_2}} &
  z_3 \ar[l]|{f_3} \ar[r]|{f_4} \ar[dll]|{f_3} \ar[dl]|{1_{z_3}} \ar[d]|{f_4} &
  z_4 \ar[dl]|{1_{z_4}} &
  x' \ar[l]|{f_5} \ar[dll]|{f_5} \ar[dl]^{1_{x'}} \\
  z_2 &
  z_3 \ar[l]|{f_3} \ar[r]|{f_4} &
  z_4 &
  x' \ar[l]|{f_5} \\
  z_3 \ar[r]|{f_4} \ar[d]_{f_4} \ar[u]^{f_3} \ar[ur]|{1_{z_3}} \ar[urr]|{f_4}  &
  z_4 \ar[dl]|{1_{z_4}} \ar[ur]|{1_{z_4}} &
  x' \ar[l]|{f_5} \ar[dll]|{f_5} \ar[dl]^{1_{x'}} \ar[u]|{f_5} \ar[ur]_{1_{x'}} \\
  z_4 &
  x' \ar[l]|{f_5} \\
  x' \ar[u]^{f_5} \ar[ur]_{1_{x'}}
}
\end{equation}
Once again, every $2$-simplex in the interior is obtained by acting on the
relevant $1$-simplex by a suitable epimorphism $[2] \to [1]$ in $\Delta$. This
completes our tour of the proof of Proposition~\ref{ssets}; again, we refer the
reader to Section~\ref{simplicialexample} for the details.

\begin{Cor}
The category of simplicial sets may be equipped with the structure of a
categorical model of identity types. In this model propositional equalities
between terms of closed type are interpreted by simplicial Moore homotopies.
\end{Cor}

\section{Homotopy-theoretic models from path object categories}\label{typecatstruct}
\looseness=-1 The purpose of this section is to prove Theorem~\ref{secondmain}:
that every path object category is a homotopy-theoretic model of identity
types. We begin by  showing that every path object category can be equipped
with the structure of a cloven w.f.s.; then we show that this cloven w.f.s.\
has a stable functorial choice of diagonal factorisations; and finally, we show
that it is functorially Frobenius.

\subsection{Cloven w.f.s.'s from path object categories}\label{sec:clovenwfs}
In this section we show that any path object category may be equipped with the
structure of a cloven w.f.s. Before doing so, we develop some notation which
will allow us to phrase our arguments in a more intuitive language.

\begin{Defn}\label{2catstructure}
Let $\E$ be a path object category. Given maps $f, g \colon X \to Y$ in $\E$,
we define a \emph{homotopy} \mbox{$\theta \colon f \Rightarrow g$} to be a
morphism $X \to MY$ which upon postcomposition with $s_X$ and $t_X$ yields $f$
and $g$ respectively. It is easy to see that the morphisms and homotopies $X
\to Y$ form a category $\underline \E(X,Y)$, whose identity and composition are
induced pointwise by the internal category structure on $MY \rightrightarrows
Y$.
%
%
%
Moreover, we have a ``whiskering'' of homotopies by morphisms: which is to say
that given a diagram
\begin{equation*}
    \cd{W \ar[r]^f & X \ar@/^12pt/[r]^g \ar@/_12pt/[r]_h \dtwocell{r}{\theta} & Y \ar[r]^k & Z\rlap{ ,}}
\end{equation*}
we have a $2$-cell $\theta.f \colon gf \Rightarrow hf$ obtained by precomposing
the corresponding map $X \to MY$ with $f$, and a $2$-cell $k.\theta \colon kg
\Rightarrow kh$ obtained by postcomposing it with $Mk$. It is easy to check
that these whiskering operations are functorial in $\theta$, and that we have
the coherence equations
\begin{equation*}
    \theta.(f.f') = (\theta.f).f'\text, \quad     \theta.1_X = \theta\text, \quad
    (k'.k).\theta = k'.(k.\theta)\text, \quad     1_Y.\theta = \theta\ \text.
\end{equation*}
Let us note also that the maps $\tau_Y$ induce a further operation on
homotopies: given $\theta \colon f \Rightarrow g \colon X \to Y$, we write
$\theta^{\mathord\circ} \colon g \Rightarrow f$ for the composite of the
corresponding map $X \to MY$ with $\tau_Y$. This reversal operation is
functorial in the obvious way.
\end{Defn}
\begin{Rk}
The above definitions determine on $\E$ a structure which is almost that of a
$2$-category, but not quite. The problem is that, given a diagram of morphisms
and homotopies
\begin{equation*}
    \cd{X \ar@/^12pt/[r]^f \ar@/_12pt/[r]_g \dtwocell{r}{\theta} & Y \ar@/^12pt/[r]^h \ar@/_12pt/[r]_k \dtwocell{r}{\phi} & Z}
\end{equation*}
in $\E$, there are two ways of composing it together---namely, as the
composites $(\phi g).(h\theta)$ and $(k \theta).(\phi f)$---and there is no
reason to expect these to agree, as they would have to in a $2$-category. The
structure we obtain is rather that of a \emph{sesquicategory} in the sense
of~\cite{Street1996Categorical}.
\end{Rk}
We now show that any path object category can be equipped with a cloven w.f.s.\
whose cloven $\ELL$-maps are ``strong deformation retracts'' in the following
sense:

\begin{Defn}\label{homotopydefn}
Let $\E$ be a path object category. By a \emph{strong deformation retraction}
for a map $f \colon X \to Y$ in $\E$, we mean a retraction $k \colon Y \to X$
for $f$ (so $kf = \id_X$) together with a homotopy $\theta \colon \id_Y
\Rightarrow fk$ which is trivial on $X$, in the sense that $\theta . f = 1_f
\colon f \Rightarrow f \colon X \to Y$. A \emph{strong deformation retract} is
a map $f$ equipped with a strong deformation retraction.
\end{Defn}

%

\begin{Prop}\label{axcloven}
Any path object category $\E$ may be equipped with the structure of a cloven
w.f.s.\ whose cloven $\ELL$-maps are the strong deformation retracts.
\end{Prop}
\begin{proof}
First, given a map $f \colon X \to Y$, we must define a factorisation $f =
\rho_f.\lambda_f$ as in~\eqref{factorisation}. So we define $Pf$ by a pullback
\begin{equation*}
   \cd{ Pf \ar[r]^{d_f} \ar[d]_{e_f} &
     MY \ar[d]^{t_Y} \\
     X \ar[r]_{f} & Y\rlap{\ \text,}
   }
\end{equation*}
define $\rho_f \defeq s_Y . d_f$,
and obtain $\lambda_f$ as the map induced by the universal property of pullback
in the following diagram:
\begin{equation*}\label{lamdef}
   \cd{
     X \ar@/^6pt/[drr]^{r_Y . f} \ar@/_6pt/[ddr]_{1_X} \ar@{.>}[dr]|{\lambda_f}
\\
     & Pf \ar[r]^{d_f} \ar[d]^{e_f} &
     MY \ar[d]^{t_Y} \\
     & X \ar[r]_{f} & Y\rlap{\ \text.}
   }
\end{equation*}
Now we have that
$\rho_f . \lambda_f = s_Y . d_f . \lambda_f = s_Y . r_Y . f = f$, so that this
yields a factorisation of $f$
as required. Next, given a square of the form~\eqref{cs}, we induce a morphism
$P(h, k) \colon Pf \to Pg$ by the universal property of pullback in the diagram
\begin{equation*}\label{Pdef}
   \cd[@!C]{
     Pf \ar[r]^{d_f} \ar[d]_{e_f} \ar@{.>}[dr]|{P(h, k)} & MV \ar[dr]^{Mk} \\
     U \ar[dr]_{h} & Pg \ar[r]_{d_g} \ar[d]_{e_g} &
     MY \ar[d]^{t_Y} \\
     & X \ar[r]_{g} & Y\rlap{\ \text.}
   }
\end{equation*}
From the commutativity of this diagram, together with naturality of $r$ and
$t$, we easily deduce the equalities $P(h, k).\lambda_f = \lambda_g . h$ and
$\rho_g . P(h, k) = k.\rho_f$. It is moreover clear that the assignation $(h,k)
\mapsto P(h,k)$ is functorial. Before continuing, let us observe that to give a
cloven $\ELL$-map $(f \colon X \to Y, s \colon Y \to Pf)$ is equally well to
give maps $f \colon X \to Y$, $k \colon Y \to X$ and $\theta \colon Y \to MY$
making the following diagrams commute:
\begin{equation*}
   \cd{ Y \ar[r]^{\theta} \ar[d]_{k} &
     MY \ar[d]^{t_Y} \\
     X \ar[r]_{f} & Y
   }\quad\text,\quad
   \cd{ Y \ar[r]^{\theta} \ar[dr]_{\id} &
     MY \ar[d]^{s_Y} \\ &
      Y
   }\quad\text,\quad
   \cd{ X \ar[r]^{f} \ar[dr]_{\id} &
     Y \ar[d]^{k} \\ &
      X
   }\quad\text,\quad
   \cd{ X \ar[r]^{f} \ar[d]_{f} &
     Y \ar[d]^{r_Y} \\ Y \ar[r]_{\theta} &
      MY
   }\ \text,
\end{equation*}
and this is precisely to give a strong deformation retract in the sense
described above. It remains only to give choices of fillers $\sigma_f$ and
$\pi_f$ as in~\eqref{sigpi}. To give $\sigma_f$ is, by the above, equally well
to give a strong deformation retraction for $\lambda_f \colon X \to Pf$. We
have $e_f \colon Pf \to X$ satisfying $e_f . \lambda_f = \id_X$, and so it
remains to give a homotopy $\theta_f \colon \id_{Pf} \Rightarrow \lambda_f .
e_f$ which is constant on $X$. So consider the following diagram:
\begin{equation}\label{theta}
   \cd[@!C]{
     Pf \ar[r]^{d_f} \ar[d]_{(M!. d_f, e_f)} \ar@{.>}[dr]|{\theta_f} & MY
\ar[dr]^{\eta_Y} \\
     M1 \times X \ar[dr]_{\alpha_{1,X}} & MPf \ar[r]_{Md_f} \ar[d]_{Me_f} &
     MMY \ar[d]^{Mt_Y} \\
     & MX \ar[r]_{Mf} & MY\ \text.
   }
\end{equation}
The inner square is a pullback, as $M$ is pullback-preserving, and we calculate
that $Mt_Y . \eta_Y . d_f = \alpha_{1, Y} . (M!, t_Y) . d_f = \alpha_{1, Y} .
(M! . d_f, f . e_f) = Mf . \alpha_{1, X} . (M! . d_f, e_f)$ so that the outer
edge commutes. Thus we induce a map $\theta_f$ as indicated; and it is now an
easy calculation to show that this map is a homotopy $\id_{Pf} \Rightarrow
\lambda_f . e_f$ that is trivial on $X$.
This completes the proof that $\lambda_f$ has a strong deformation retraction,
and hence the construction of $\sigma_f$. We now turn to $\pi_f$. Let us
consider the diagram
\begin{equation*}\label{jdef}
   \cd[@+0.4em@C-0.4em]{
     P\rho_f \ar@/^8pt/[drr]^{d_{\rho_f}} \ar@/_6pt/[ddr]_{d_f . e_{\rho_f}}
\ar@{.>}[dr]|{j_f} \\
     & MY \times_Y MY \ar[r]_-{\pi_2} \ar[d]^{\pi_1} &
     MY \ar[d]^{t_Y} \\
     & MY \ar[r]_{s_Y} & Y\rlap{\ \text.}
   }
\end{equation*}
We have that $t_Y . d_{\rho_f} = \rho_f . e_{\rho_f} = d_f . s_Y . e_{\rho_f}$
so that the outside of this diagram commutes, and hence we induce a unique
filler $j_f$ as indicated. We now consider the diagram
\begin{equation*}\label{pidef}
   \cd{
     P\rho_f \ar[r]^-{j_f} \ar[d]_{e_{\rho_f}} \ar@{.>}[dr]|{\pi_f} & MY
\times_Y MY \ar[dr]^{m_Y} \\
     Pf \ar[dr]_{e_f} & Pf \ar[r]_-{d_f} \ar[d]^{e_f} &
     MY \ar[d]^{t_Y} \\
     & X \ar[r]_{f} & Y\rlap{\ \text.}
   }
\end{equation*}
We calculate that $t_Y . m_Y . j_f = t_Y . \pi_1 . j_f = t_Y . d_f . e_{\rho_f}
= f . e_f . e_{\rho_f}$, so that
 the outside of this diagram commutes;
and so we induce a unique filler $\pi_f$ as indicated. We must now show that
this $\pi_f$ renders commutative both triangles in the right-hand square
of~\eqref{sigpi}. For the lower-right triangle, we have that $\rho_f . \pi_f =
s_Y . d_f . \pi_f = s_Y . m_Y . j_f = s_Y . \pi_2 . j_f = s_Y . d_{\rho_f} =
\rho_{\rho_f}$ as required. For the upper-left triangle, it suffices to show
that $\pi_f . \lambda_{\rho_f} = 1_{Pf}$ holds upon postcomposition with $e_f$
and $d_f$. In the former case, we have that $e_f . \pi_f . \lambda_{\rho_f} =
e_f . e_{\rho_f} . \lambda_{\rho_f} = e_f$; whilst in the latter, we calculate
that $ d_f . \pi_f . \lambda_{\rho_f} = m_Y . j_f . \lambda_{\rho_f} = m_Y .
(d_f.e_{\rho_f}, d_{\rho_f}).\lambda_{\rho_f} = m_Y . (d_f, r_Y . \rho_f) =
m_Y. (1_{MY},r_Y.s_Y).d_f = d_f $ as required. This completes the construction
of $\pi_f$.%
\end{proof}

We may also give an explicit characterisation of the cloven $\R$-maps of the
cloven w.f.s.\ constructed in Theorem~\ref{axcloven}, since by the Yoneda
lemma, to equip a map $f \colon X \to Y$ with a cloven $\R$-map structure $p
\colon Pf \to X$ is equally well to give a natural family of functions $\E(V,p)
\colon \E(V, Pf) \to \E(V,X)$ rendering commutative all diagrams of the form
\begin{equation}\label{commutepathlift}
    \cd[@+0.5em]{
        \E(V,X) \ar[d]_{\E(V,\lambda_f)} \ar[r]^-{\id} & \E(V,X) \ar[d]^{\E(V,f)} \\
        \E(V,Pf) \ar[ur]_{\E(V,p)} \ar[r]_-{\E(V,\rho_f)} & \E(V,Y)\rlap{ .}
    }
\end{equation}
Now, to give a morphism $V \to Pf$ is, by definition of $Pf$, equally well to
give morphisms $\phi \colon V \to MY$ and $x \colon V \to X$ satisfying $f.x =
t_Y.\phi$;
which in the notation of Definition~\ref{homotopydefn}, is equally well to give
morphisms $x \colon V \to X$ and $y \colon V \to Y$ together with a homotopy
$\phi \colon y \Rightarrow fx$. So to give the function $\E(V,p)$ is to assign
to every such collection of data a morphism which we might suggestively write
as $\phi^\ast(x) \colon V \to X$; to ask for commutativity of the two triangles
in~\eqref{commutepathlift} is to ask, firstly, that $f.\phi^\ast(x) = y$, and
secondly, that when $\phi$ is an identity homotopy $fx \Rightarrow fx$, we have
$\phi^\ast(x) = x$; whilst to ask for naturality in $V$ is to ask that for
every $g \colon W \to V$, we have $\phi^\ast(x)g = (\phi g)^\ast(xg)$.

This resembles a path-lifting property: it says that if we have a
$V$-parameterised path in $Y$, together with a lifting of its target to $X$, we
can also lift its source to $X$. What it does not say, a priori, is that we can
lift the entire path $\phi$ to $X$. But in fact, we can derive this by making
use of the path contractions obtained from $\eta$. Indeed, we have a map $\ell
\colon Pf \to MX$ given by the composite $Mp.\theta_f$,
where $\theta_f$ is given as in~\eqref{theta}. Straightforward calculation now
shows that $s_X . \ell = p$, that $t_X . \ell = e_f$, and that $Mf . \ell =
d_f$; which says that if we are given a morphism $V \to Pf$, corresponding as
before to a homotopy $\phi \colon y \to fx$, then postcomposition with $\ell$
yields a homotopy $\bar \phi \colon \phi^\ast(x) \Rightarrow x \colon V \to X$
such that $f.\bar \phi = \phi$. We record the content of this discussion as:

\begin{Prop}\label{pathlift}
In the cloven w.f.s.\ of Theorem~\ref{axcloven}, cloven $\R$-map structures on
$f \colon X \to Y$ are in bijective correspondence with operations which, to
every morphism $x \colon V \to X$ and homotopy $\phi \colon y \Rightarrow fx
\colon V \to Y$, assign a homotopy $\bar \phi \colon \phi^\ast(x) \Rightarrow x
\colon V \to X$ such that $f.\bar \phi = \phi$, such that $\bar \phi$ is the
identity homotopy whenever $\phi$ is, and such that for any map $g \colon W \to
V$, we have $(\phi g)^\ast(xg) = \phi^\ast(x)g$ and $\overline{\phi g} = \bar
\phi. g$.
\end{Prop}

\subsection{Factorising the diagonal}\label{facdiag} To show that the cloven
w.f.s.\ associated to a path object category underlies a homotopy-theoretic
model of identity types, there remain two tasks: firstly, to show that it has a
stable functorial choice of diagonal factorisations, and secondly, to show that
it is functorially Frobenius. In this section we establish the first of these.
The intuition behind our construction is that it should be enough to establish
the existence of the required diagonal factorisations in the case of a cloven
$\R$-map $X \to 1$---that is, a ``closed type''---as we can then obtain it for
an arbitrary $\R$-map $X \to \Gamma$---a ``dependent type''---by regarding it
as a ``closed type'' in the slice category $\E / \Gamma$. This then ensures the
required pullback-stability of the factorisations we construct.

Now, in order for this intuition to make sense, we must know that the notion of
path object structure is an indexed one: which is to say that such a structure
on $\E$ induces a corresponding one on every slice $\E / \Gamma$ in a
pullback-stable manner. This is a consequence of a result of Robert Par\'e
(see~\cite[Proposition 3.3]{Johnstone1997Cartesian}), which states that any
pullback-preserving, strong endofunctor of a finitely complete $\E$ extends to
an $\E$-indexed endofunctor; and that any strong natural transformation between
two such extends to an $\E$-indexed natural transformation. This applies in
particular to the endofunctor $M$ and the natural transformations $s, t, r, m,
\tau$ and $\eta$ of our axiomatic framework, so that the notion of path object
category is indeed an indexed one. For the sake of a self-contained
presentation, we now reproduce such details of Par\'e's result as are necessary
in what follows.
\begin{Prop}\label{subcat}
Suppose given a map $x \colon X \to \Gamma$ in a path object category, and
consider the following pullback diagram:
\begin{equation*}
    \cd{
        M_{\Gamma}(x) \ar[d]_{u_x} \ar[r]^-{j_x} &
        MX \ar[d]^{Mx} \\
        M1 \times \Gamma \ar[r]_-{\alpha_{1,\Gamma}} &
        M\Gamma\rlap{\ .}
    }
\end{equation*}
If we define $s_x$ and $t_x \colon M_\Gamma(x) \to X$ as the respective
composites $s_X.j_x$ and $t_X.j_x$, then the subgraph $(s_x,t_x) \colon
M_\Gamma(x) \rightrightarrows X$ of $(s_X,t_X) \colon MX \rightrightarrows X$
is a subcategory.
\end{Prop}
\begin{proof}
Informally, $M_\Gamma(x)$ is the subobject of $MX$ consisting of all those
paths in $X$ which are sent by $x$ to a constant path in $\Gamma$; and thus we
need to prove that identity paths are constant, and that constant paths are
closed under composition. To make this formal, we observe that the square
\begin{equation*}
    \cd[@C+3em]{
        M_\Gamma(x) \rightrightarrows X \ar[r]^-{(j_x, 1_X)} \ar[d]_{(u_x, x)} &
MX \rightrightarrows X \ar[d]^{(Mx, x)} \\
        M1 \times \Gamma \rightrightarrows \Gamma \ar[r]_-{(\alpha_{1,\Gamma},
1_\Gamma)} & M\Gamma \rightrightarrows \Gamma\ \text.
    }
\end{equation*}
is a pullback in the category of graphs internal to $\E$. So to prove the
statement of the Proposition,  it suffices to show that the subgraph $(\pi_2,
\pi_2) \colon M1 \times \Gamma \rightrightarrows \Gamma$ of $(s_\Gamma,
t_\Gamma) \colon M\Gamma \rightrightarrows \Gamma$ is a subcategory. For this,
we must show two things: firstly, that $r_\Gamma \colon \Gamma \to M\Gamma$
factors through $\alpha_{1,\Gamma} \colon M1 \times \Gamma \to M\Gamma$---which
is immediate by the strength of $r$, since we have $r_\Gamma =
\alpha_{1,\Gamma} . (r_1 \times \Gamma)$---and secondly,
that the composite
\begin{equation*}
    (M1 \times \Gamma) \times_\Gamma (M1 \times \Gamma)
\xrightarrow{\alpha_{1,\Gamma} \times_\Gamma \alpha_{1,\Gamma}} M\Gamma
\times_\Gamma M\Gamma \xrightarrow{m_\Gamma} M\Gamma
\end{equation*}
factors through $\alpha_{1,\Gamma} \colon M1 \times \Gamma \to M\Gamma$. But we
observe that the domain of this composite is isomorphic to $M1 \times M1 \times
\Gamma$, and that upon composition with this isomorphism, the map displayed
above becomes
\begin{equation*}
    M1 \times M1 \times \Gamma \xrightarrow{(\alpha_{1,\Gamma}.(\pi_1, \pi_3),
\alpha_{1,\Gamma}.(\pi_2, \pi_3))} M\Gamma \times_\Gamma M\Gamma
\xrightarrow{m_\Gamma} M\Gamma\ \text.
\end{equation*}
But this map factors through $\alpha_{1,\Gamma}$ since, by the strength of $m$,
the following diagram commutes:
\begin{equation*}
  \cd[@C+10em]{
    M1 \times M1 \times \Gamma \ar[r]^{(\alpha_{1,\Gamma}.(\pi_1, \pi_3),
\alpha_{1,\Gamma}.(\pi_2, \pi_3))} \ar[d]_{m_1 \times \Gamma} &
    M\Gamma \times_\Gamma M\Gamma \ar[d]^{m_\Gamma} \\
    M1 \times \Gamma \ar[r]_{\alpha_{1,\Gamma}} & M\Gamma\rlap{\ \text.}
  }
\end{equation*}
Hence $M1 \times \Gamma \rightrightarrows \Gamma$ is a subcategory of $M\Gamma
\rightrightarrows \Gamma$, and so $M_\Gamma(x) \rightrightarrows X$ is a
subcategory of $MX \rightrightarrows X$ as desired.
\end{proof}
Given a map $x \colon X \to \Gamma$ of our path object category $\E$, we will
denote the structure maps of the internal category $(s_x,t_x) \colon
M_\Gamma(x) \rightrightarrows X$ by $r_x \colon X \to M_\Gamma(x)$ and $m_x
\colon M_\Gamma(x) \times_X M_\Gamma(x) \to M_\Gamma(x)$; observe that we have
$j_x . r_x = r_X$ and $j_x . m_x = m_X . (j_x \times_X j_x)$. With this in
hand, we are now ready to prove:
\begin{Prop}\label{choiceofdiagonalfacs}
In a path object category $\E$, the assignation sending a cloven $\R$-map
$(x,p) \colon X \to \Gamma$ to the factorisation
\begin{equation*}
    X \xrightarrow{r_x} M_\Gamma(x) \xrightarrow{(s_x,t_x)} X \times_\Gamma X
\end{equation*}
yields a choice of diagonal factorisations in the sense of
Definition~\ref{stablepath}(i).
\end{Prop}
\begin{proof}
We must show that each $r_x$ may be made into a cloven $\ELL$-map and each
$(s_x,t_x)$ into a cloven $\R$-map. Now, to make $r_x$ into a cloven $\ELL$-map
is equally well to equip it with a strong deformation retraction in the sense
of Definition~\ref{homotopydefn}. A suitable retraction for $r_x$ is given by
$t_x \colon M_\Gamma(x) \to X$, since $t_x.r_x = t_X.j_x.r_x = t_X.r_X =
\id_X$; so we now need a homotopy $\id_{M_\Gamma(x)} \Rightarrow r_x.t_x$ which
is constant on $X$. To construct this, consider the following diagram:
\begin{equation*}
   \cd[@!C]{
     M_\Gamma(x) \ar[r]^{j_x} \ar[d]_{u_x} \ar@{.>}[dr]^{\varphi} & MX \ar[dr]^{\eta_X} \\
     M1 \times \Gamma \ar[dr]_{\alpha_{M1,\Gamma}.(\eta_1 \times \Gamma)\ \ \ } & MM_\Gamma(x) \ar[r]_{Mj_x} \ar[d]_{Mu_x} &
     MMX \ar[d]^{MMx} \\
     & M(M1 \times \Gamma) \ar[r]_-{M\alpha_{1,\Gamma}} & MM\Gamma\ \text.
   }
\end{equation*}
Its outside commutes by naturality and strength of $\eta$; and since the
bottom-right square is a pullback, we induce a unique morphism $\varphi \colon
M_\Gamma(x) \to MM_\Gamma(x)$ as indicated. Now an easy calculation shows that
this $\varphi$ gives rise to a homotopy $\id_{M_\Gamma(x)} \Rightarrow r_x.t_x$
which is trivial on $X$,
so that $r_x$ is a cloven $\ELL$-map as required. We must now show that
$(s_x,t_x) \colon M_\Gamma(x) \to X \times_\Gamma X$ is a cloven $\R$-map; for
which it suffices, by the discussion preceding Proposition~\ref{pathlift}, to
define an operation which to every $V$-indexed homotopy in $X \times_\Gamma X$
and every lifting of its target to $M_\Gamma(x)$, associates a lift of its
source to $M_\Gamma(x)$. Now a $V$-indexed homotopy in $X \times_\Gamma X$ is a
pair of homotopies $\phi_1 \colon a' \Rightarrow a \colon V \to X$ and $\phi_2
\colon b' \Rightarrow b \colon V \to X$ such that $x.\phi_1 = x. \phi_2$; let
us write this common composite as $\psi \colon c' \Rightarrow c \colon V \to
\Gamma$. The target of $(\phi_1, \phi_2)$ is the pair $(a, b)$; and to give a
lifting of this to $M_\Gamma(x)$ is to give a homotopy $\chi \colon a
\Rightarrow b \colon V \to X$ such that $x.\chi$ is a \emph{constant} homotopy
$c \Rightarrow c$; that is, one for which the corresponding map $V \to M\Gamma$
factors through $M1 \times \Gamma$. From these data we are required to
determine a homotopy $\chi' \colon a' \Rightarrow b' \colon V \to X$ for which
$x.\chi'$ is a homotopy constant at $c'$. We may represent this
diagrammatically as follows:
\begin{equation}\label{diagrammatic}
\cd[@!@-2em]{
  a' \ar@2[rr]^{\phi_1} \ar@{.>}[dd]_{\chi'} & &
  a \ar@2[dd]^{\chi} \\ & & & & X \ar@{-->}[ddd]^x \\
  b' \ar@2[rr]_{\phi_2} & &
  b \\ & & & & \\
  c' \ar@2[rr]_{\psi} & &
  c & & \Gamma\rlap{ .}
 }
\end{equation}
We will construct $\chi'$ as the composite of three homotopies constant over
$c'$ which we obtain by means of the following two lemmas. The first describes
a ``cartesianness'' property of the path-liftings obtained from a cloven
$\R$-map.
\begin{Lemma}\label{precart}
Let $(x,p) \colon X \to \Gamma$ be a cloven $\R$-map. Then to each $V \in \E$
and homotopy $\xi \colon y \Rightarrow z \colon V \to X$ we may associate a
homotopy $\tilde \xi \colon y \Rightarrow (x.\xi)^\ast(z)$ constant over $xy$,
such that for all $f \colon W \to V$ we have $\widetilde{\xi f} = \tilde \xi.
f$, and such that when $\xi$ is an identity homotopy, so is $\tilde \xi$.
\end{Lemma}
The second describes a ``functoriality'' property of these same path-liftings.
\begin{Lemma}\label{functorialityprop}
Let $(x,p) \colon X \to \Gamma$ be a cloven $\R$-map. Then to each $V \in \E$,
each homotopy $\psi \colon c' \Rightarrow c \colon V \to \Gamma$, and each
homotopy $\xi \colon y \Rightarrow z \colon V \to X$ constant over $c$, we may
associate a homotopy $\psi^\ast(\xi) \colon \psi^\ast(y) \Rightarrow
\psi^\ast(z) \colon V \to X$ constant over $c'$, in such a way that for all $f
\colon W \to V$ we have $(\psi f)^\ast(\xi f) = (\psi^\ast(\xi)).f$ and such
that when $\psi$ is an identity homotopy, $\psi^\ast(\xi) = \xi$.
\end{Lemma}
If we leave aside the proof of these two lemmas for a moment, then we may
construct $\chi'$ as the composite homotopy
\begin{equation*}
    a' \xRightarrow{\widetilde{\phi_1}} \psi^\ast(a) \xRightarrow{\psi^\ast(\chi)} \psi^\ast(b) \xRightarrow{(\widetilde{\phi_2})^\circ} b'\rlap{ .}
\end{equation*}
Observe that this is a composite of homotopies constant over $c'$, and hence by
Proposition~\ref{subcat}, is itself a homotopy constant over $c'$. In order to
conclude that this assignation yields a cloven $\R$-map structure on
$(s_x,t_x)$, we must verify two things. Firstly, that the operation $\chi
\mapsto \chi'$ is natural in $V$---which follows immediately from the
corresponding naturalities noted in Lemmas~\ref{precart}
and~\ref{functorialityprop}---and secondly, that when $\phi_1$ and $\phi_2$ are
identity homotopies, we have $\chi' = \chi$. But in this situation the
displayed composite reduces, again by Lemmas~\ref{precart}
and~\ref{functorialityprop}, to the composite
\begin{equation*}
    a \xRightarrow{1_a} a \xRightarrow{\chi} b \xRightarrow{1_b} b\rlap{ ,}
\end{equation*}
which by left and right unitality of composition, is equal to $\chi$. Note
that, as anticipated in Remark~\ref{allowableweakening}, this is the only point
in the whole argument where we make use of the right unitality of composition.
\end{proof} It remains to give
the proof of the above two Lemmas.
\begin{proof}[of Lemma~\ref{precart}]
By the Yoneda lemma, to give the indicated assignation is equally well to give
a morphism $\delta \colon MX \to M_\Gamma(x)$ such that:
\begin{equation}\label{theeqs}
s_x . \delta = s_X\text, \qquad
t_x . \delta = p . \xi\text, \qquad \text{and} \qquad
\delta . r_X = r_x\text,
\end{equation}
where $\xi \colon MX \to Px$ is the morphism induced by the universal property
of pullback in the following diagram:
\begin{equation*}
   \cd{
     MX \ar@/^6pt/[drr]^{Mx} \ar@/_6pt/[ddr]_{t_X} \ar@{.>}[dr]|{\xi} \\
     & Px \ar[r]^{d_x} \ar[d]^{e_x} &
     M\Gamma \ar[d]^{t_\Gamma} \\
     & X \ar[r]_{x} & \Gamma\rlap{\ \text.}
   }
\end{equation*}
Now, postcomposition with $p.\xi \colon MX \to X$ takes a homotopy $\psi \colon
z \Rightarrow w \colon V \to X$ and sends it to $(x.\psi)^\ast(w) \colon V \to
X$. In particular, when $\psi$ is the identity homotopy on $z$, applying
$p.\xi$ gives back $z$ again. Therefore we can construct the required homotopy
$z \Rightarrow (f\psi)^\ast(w)$ by first forming the following ``path of
paths''
\begin{equation*}
    \cd[@!@+0.5em]{
    z \ar[r] \ar[d]_{r_X(z)} \rtwocell{dr}{} &
    z \ar[d]^{\psi} \\
    z \ar[r] & w
    }
\end{equation*}
in $MMX$ using $\eta$ and the involution $\tau$, and then applying $M(p.\xi)$
to it. Formally, we consider the following composite:
\begin{equation*}
    MX \xrightarrow{\tau_X} MX \xrightarrow{\eta_X} MMX \xrightarrow{\tau\tau_X}
MMX \xrightarrow{M\xi} MPx \xrightarrow{Mp} MX\ \text.
\end{equation*}
We claim that this factors through the subobject $M_\Gamma(x)$ of $MX$, so that
we may define $\delta$ to be the factorising map. To prove the claim, we must
show that $Mx.\delta$ factors through $\alpha_{1,\Gamma}$, for which we
calculate that
\begin{align*}
    Mx . \delta &
     = M\rho_x . M \xi . {\tau\tau}_X . \eta_X . \tau_X
     = Ms_\Gamma . MMx . {\tau\tau}_X . \eta_X . \tau_X\\&
     = Mx . Ms_X . {\tau\tau}_X . \eta_X . \tau_X
     = Mx . \tau_X . Mt_X . \eta_X . \tau_X\\&
     = Mx . \tau_X . \alpha_{1,X} . (M!, t_X) . \tau_X
     = Mx . \alpha_{1,X} . (\tau_1 . M!, t_X).\tau_X\\&
     = \alpha_{1,\Gamma} . (\tau_1 . M!, x . t_X).\tau_X
\end{align*}
as required. It is now straightforward to check that this $\delta$ satisfies
the equations in~\eqref{theeqs}.
\end{proof}
\begin{proof}[of Lemma~\ref{functorialityprop}]
Recall that, given a homotopy $\psi \colon c' \Rightarrow c \colon V \to
\Gamma$ and a homotopy $\xi \colon y \Rightarrow z \colon V \to X$ constant
over $c$, we are required to produce a homotopy $\psi^\ast(\xi) \colon
 \psi^\ast(y) \to \psi^\ast(z)$ constant over $c'$. We
shall do so by first defining a homotopy $\theta \colon u \Rightarrow v \colon
V \to Px$ which we schematically depict as:
\begin{equation}\label{schematic}
u = \left(\cd{
   c \ar@2[d]_{\psi} \\ 
   xy 
   }\right)
\cd{
   {} \ar@2[r]^{\text{constant}} & {} \\
   {} \ar@2[r]_-{x\xi} & {}
} \left(\cd{
   c \ar@2[d]^{\psi} \\
   xz
 }\right) = v\rlap{ .}
\end{equation}
Thus $u$ will correspond to the map $V \to Px$ picking out the homotopy $\psi$
together with the lifting $y$ of its endpoint, whilst $v$ will correspond to
the map picking out $\psi$ together with the lifting $z$ of its endpoint. It
follows that postcomposing $\theta$ with $p \colon Px \to X$ yields the desired
homotopy $\psi^\ast(\xi) \colon \psi^\ast(y) \Rightarrow \psi^\ast(z) \colon V
\to X$.

Now, to give the homotopy $\theta$ is to give a map $V \to MPx$ which---since
$M$ preserves pullbacks---is equally well to give maps $\gamma$ and $\delta$
rendering commutative the square
\begin{equation}\label{tocommute}
\cd{
    V \ar[r]^-{\delta} \ar[d]_\gamma & MM\Gamma \ar[d]^{Mt_\Gamma} \\
    MX \ar[r]_-{Mx} & M\Gamma\rlap{ .}
}
\end{equation}
To define $\gamma$ and $\delta$, let us first note that to give the homotopies
$\psi$ and $\xi$ is equally well to give maps---which by abuse of notation, we
also denote by $\psi$ and~$\xi$---rendering the following square commutative:
\begin{equation*}
  \cd{
    V \ar[r]^-\psi \ar[d]_\xi &
    M\Gamma \ar[d]^{t_\Gamma} \\
    M_\Gamma(x) \ar[r]_-{\pi_2.u_x} & \Gamma\rlap{ .}
  }
\end{equation*}
Now by referring to our schematic depiction~\eqref{schematic}, we see that we
should take $\gamma = j_x.\xi$, and take $\delta$ to be the composite
\begin{equation*}
    V \xrightarrow{(\pi_1.u_x.\xi,\,\psi)} M1 \times M\Gamma \xrightarrow{\alpha_{1,M\Gamma}} MM\Gamma
\end{equation*}
picking out the homotopy in $M\Gamma$ ``constant at $\psi$ and of the same
length as $\xi$''. We now calculate that $Mt_\Gamma.\delta =
Mt_\Gamma.\alpha_{1,M\Gamma}.(\pi_1.u_x.\xi,\psi) =
\alpha_{1,\Gamma}.(\pi_1.u_x.\xi,t_\Gamma.\psi) =
\alpha_{1,\Gamma}.(\pi_1.u_x.\xi,\pi_2.u_x.\xi) = \alpha_{1,\Gamma}.u_x.\xi =
Mx.j_x.\xi = Mx.\gamma$, so that~\eqref{tocommute} commutes; and thus we have
an induced morphism $(\gamma,\delta) \colon V \to MPx$ defining the desired
homotopy ${\theta \colon u \Rightarrow v \colon V \to Px}$. Observe that $u$
and $v$ are the morphisms $V \to Px$ induced by the universal property of
pullback in the respective diagrams
\begin{equation*}
\cd{
    V \ar[r]^-{s_{M\Gamma}.\delta} \ar[d]_{s_X.\gamma} & M\Gamma \ar[d]^{t_\Gamma} \\
    X \ar[r]_-{x} & \Gamma
} \qquad \text{and} \qquad
\cd{
    V \ar[r]^-{t_{M\Gamma}.\delta} \ar[d]_{t_X.\gamma} & M\Gamma \ar[d]^{t_\Gamma} \\
    X \ar[r]_-{x} & \Gamma
}\ \text;
\end{equation*}
and we calculate that $s_{M\Gamma}.\delta =
s_{M\Gamma}.\alpha_{1,M\Gamma}.(\pi_1.u_x.\xi,\psi) =
\pi_2.(\pi_1.u_x.\xi,\psi) = \psi$ and that $s_X.\gamma = s_X.j_x.\xi = s_x.\xi
= y$; similarly that $t_{M\Gamma}.\delta = \psi$ and that $t_X.\gamma = z$.
Hence if we define $\psi^\ast(\xi)$ to be the composite $Mp.\theta$, then we
obtain a homotopy $\psi^\ast(y) \Rightarrow \psi^\ast(z) \colon V \to X$ as
required. It remains to check three things. Firstly, we must show that
$\psi^\ast(\xi)$ is a constant homotopy; it will
    then necessarily be over $c'$, since $\psi^\ast(y)$ and $\psi^\ast(z)$
    are. For this, we must show that the composite morphism
    $Mx.Mp.(\gamma,\delta) \colon V \to M\Gamma$ factors through
    $\alpha_{1,\Gamma}$; and so we calculate that $Mx.Mp.(\gamma,\delta) =
    M\rho_x.(\gamma,\delta) = Ms_\Gamma.Md_x.(\gamma,\delta) =
    Ms_\Gamma.\delta = Ms_\Gamma.\alpha_{1,M\Gamma}.(\pi_1.u_x.\xi,\psi) =
    \alpha_{1,\Gamma}.(\pi_1.u_x.\xi,s_\Gamma.\psi)$ as required.

Secondly, we must show that for any $f \colon W \to V$, we have $(\psi
    f)^\ast(\xi f) = (\psi^\ast(\xi))f$. This is immediate by the universal
    property of pullback. Finally, we must show that
    when $\psi$ is an identity homotopy, we have $\psi^\ast(\xi) = \xi$. In
    this case, we have that $\psi = r_\Gamma.t_\Gamma.\psi \colon V \to
    M\Gamma$; and are required to show that $Mp.(\gamma,\delta) = j_x.\xi
    \colon V \to MX$. We calculate that
\begin{align*}
\delta &= \alpha_{1,M\Gamma}.(\pi_1.u_x.\xi,\psi)
  = \alpha_{1,M\Gamma}.(\pi_1.u_x.\xi,r_\Gamma.t_\Gamma.\psi)
  = \alpha_{1,M\Gamma}.(\pi_1.u_x.\xi,r_\Gamma.\pi_2.u_x.\xi)\\&
  = Mr_{\Gamma}.\alpha_{1,\Gamma}.(\pi_1.u_x.\xi,\pi_2.u_x.\xi)
  =  Mr_{\Gamma}.\alpha_{1,\Gamma}.u_x.\xi
  = Mr_{\Gamma}.Mx.j_x.\xi\ \text,
\end{align*}
and so have that $Mp.(\gamma,\delta) = Mp.(\id_{MX},Mr_\Gamma.Mx).j_x.\xi =
    Mp.M\lambda_x.j_x.\xi = j_x.\xi$ as required.
\end{proof}
This completes the proof of the two Lemmas, and hence of
Proposition~\ref{choiceofdiagonalfacs}. So we have established the existence of
a choice of diagonal factorisations for the cloven w.f.s.\ associated to any
path object category; the crucial point is that this choice is a suitably
well-behaved one.
\begin{Prop}\label{isfunctorial}
For any path object category $\E$, the choice of diagonal factorisations on its
associated cloven w.f.s.\ given by Proposition~\ref{choiceofdiagonalfacs} is
functorial and stable in the sense of Definition~\ref{stablepath}(ii)--(iii).
\end{Prop}
\begin{proof}
To show functoriality, observe first that if $\E$ is a path object category,
and $\C$ any category, then the functor category $[\C, \E]$ becomes a path
object category with the structure given pointwise; and that the category of
$\R$-maps in $[\C, \E]$ is then $[\C, \R\text-\cat{Map}_\E]$, and
correspondingly for the category of $\ELL$-maps. In particular, given any path
object category $\E$, the path object category $[\R\text-\cat{Map}_\E, \E]$
contains a ``generic $\R$-map'', corresponding to the identity functor
$\R\text-\cat{Map}_\E \to \R\text-\cat{Map}_\E$. Applying
Proposition~\ref{choiceofdiagonalfacs} to this generic $\R$-map yields the
desired functoriality.

It remains to show stability: for which it suffices to show that, given a
cloven $\R$-map $x \colon X \to \Gamma$ together with a morphism $f \colon
\Delta \to \Gamma$, there is an isomorphism $f^\ast(M_\Gamma(x)) \to M_\Delta
(f^\ast x)$ over $\Delta$ which is compatible with $r$, $s$ and $t$ in the
obvious way. Now, $f^\ast(M_\Gamma(x))$ is isomorphic to the composite down the
left of the following diagram, all of whose squares are pullbacks:
\begin{equation*}
\cd[@C+0.5em]{
    f^\ast(M_\Gamma(x)) \ar[r] \ar[d] &
    M_\Gamma(x) \ar[r]^{j_x} \ar[d]_{u_x} &
    MX \ar[d]^{Mx} \\
    M1 \times \Delta \ar[r]_-{M1 \times f} \ar[d]_{\pi_2} &
    M1 \times \Gamma \ar[r]_-{\alpha_{1,\Gamma}} \ar[d]^{\pi_2} &
    M \Gamma\rlap{ .}\\
    \Delta \ar[r]_f &
    \Gamma}
\end{equation*}
On the other hand, $M_\Delta(f^\ast x)$ is obtained as the composite with
$\pi_2$ of the arrow down the left of the following diagram, all of whose
squares are again pullbacks
\begin{equation*}
\cd[@C+0.5em]{
    M_\Delta (f^\ast X) \ar[r] \ar[d] &
    M(f^\ast X) \ar[r] \ar[d]_{M(f^\ast x)} &
    MX \ar[d]^{Mx} \\
    M1 \times \Delta \ar[r]_-{\alpha_{1,\Delta}} &
    M\Delta \ar[r]_-{Mf}&
    M \Gamma\rlap{ .}}
\end{equation*}
But since $Mf . \alpha_{1,\Delta} = \alpha_{1, \Gamma} . (M1 \times f)$ by
naturality of $\alpha$, we have these two composites isomorphic to each other
over $\Delta$ as required. The naturality of these isomorphisms, and their
coherence with the remaining data, is easily checked.
\end{proof}

\subsection{Functorial Frobenius structure}\label{sec:functorialfrob} We have now verified that any
path object category can be equipped with a cloven w.f.s.\ which has a stable
functorial choice of diagonal factorisations. The final step in proving
Theorem~\ref{secondmain} is to show that this cloven w.f.s.\ is functorially
Frobenius.
\begin{Prop}\label{frobconstr}
If $\E$ is a path object category, then its associated cloven w.f.s.\ is
functorially Frobenius in the sense of Definition~\ref{stablepath}(iv).
\end{Prop}
\begin{proof}
Let $(f,p) \colon B \to A$ be a cloven $\R$-map and $(i,q) \colon X \to A$ a
cloven $\ELL$-map in $\E$, and consider a pullback square
\begin{equation*}
\cd{
  f^\ast X \ar[r]^{\bar f} \ar[d]_{\bar \imath} & X \ar[d]^i \\
  B \ar[r]_f & A\rlap{ .}
}
\end{equation*}
We are required to equip $\bar \imath$ with the structure of a cloven
$\ELL$-map. Recall from Theorem~\ref{axcloven} that to give a cloven $\ELL$-map
structure on $i$ is to give a strong deformation retraction for it: thus a map
$k \colon A \to X$ satisfying $ki = \id_X$ and a homotopy $\theta \colon 1_A
\Rightarrow ik \colon A \to A$ such that $\theta.i = 1_i \colon i \Rightarrow
i$. Correspondingly, to construct a cloven $\ELL$-map structure on $\bar
\imath$ it suffices to give a strong deformation retraction for it. So consider
the homotopy $\phi = (\theta f)^\circ \colon ikf \Rightarrow f \colon B \to A$,
where we recall from Remark~\ref{2catstructure} that $(\thg)^\circ$ is the
reversal operation on homotopies induced by $\tau$. Since $f$ is a cloven
$\R$-map, we may apply its path-lifting property, described in
Proposition~\ref{pathlift}, to obtain from this a morphism $\phi^\ast(1_B)
\colon B \to B$ and a homotopy $\bar \phi \colon \phi^\ast(1_B) \Rightarrow 1_B
\colon B \to B$ such that $f.\bar \phi = \phi$. In particular, this means that
$f . \phi^\ast(1_B) = ikf$, and so we may induce a morphism $\bar k \colon B
\to f^\ast X$ by the universal property of pullback in
\begin{equation*}
   \cd{
     B \ar@/^6pt/[drr]^{kf} \ar@/_6pt/[ddr]_{\phi^\ast(1_B)} \ar@{.>}[dr]|{\bar k}\\
     & f^\ast X \ar[r]^{\bar f} \ar[d]^{\bar \imath} &
     X \ar[d]^{i} \\
     & B \ar[r]_{f} & A\ \text.
   }
\end{equation*}
Now taking $\bar \theta = (\bar \phi)^\circ \colon 1_B \Rightarrow \bar \imath
\bar k$, then it remains to show that $\bar k. \bar \imath = 1_{f^\ast B}$ and
that $\bar \theta. \bar \imath$ is the identity homotopy on $\bar \imath$.
Firstly, to show that $\bar k. \bar \imath = 1_{f^\ast B}$, it suffices to
demonstrate equality on postcomposition with $\bar f$ and with $\bar \imath$.
On the one hand, we have that $\bar f. \bar k . \bar \imath = k.f.\bar \imath =
k.i.\bar f = \bar f$; whilst on the other, we have $\bar \imath.\bar k.\bar
\imath = \phi^\ast(1_B).\bar \imath = (\phi.\bar \imath)^\ast(\bar \imath)$ by
naturality of the liftings described in Proposition~\ref{pathlift}. But
$\phi.\bar \imath = (\theta f)^\circ.\bar \imath = (\theta f \bar \imath)^\circ
= (\theta i \bar f)^\circ = (1_{i \bar f})^\circ = 1_{i \bar f}$, and so by
Proposition~\ref{pathlift} again we have $(\phi.\bar \imath)^\ast(\bar \imath)
= \bar \imath$ as required. Thus $\bar k. \bar \imath = 1_{f^\ast B}$, and it
remains only to show that $\bar \theta. \bar \imath = 1_{\bar\imath}$. Applying
$(\thg)^\circ$ to both sides, it suffices to show that $\bar \phi. \bar \imath
= 1_{\bar \imath}$. But by naturality of the liftings of
Proposition~\ref{pathlift}, we have $\bar \phi. \bar \imath =
\overline{\phi.\bar \imath}$, and since $\phi.\bar \imath$ is an identity
homotopy as above, we deduce that $\bar \phi.\bar \imath$ is as well.

Thus we have equipped $\bar \imath$ with the structure of a strong deformation
retract, and hence of a cloven $\ELL$-map, which completes the verification
that the cloven w.f.s.\ associated to $\E$ is Frobenius. Finally, to show that
it is \emph{functorially} Frobenius we apply an entirely analogous argument to
that given in Proposition~\ref{isfunctorial}.
\end{proof}
This now completes the proof of Theorem~\ref{secondmain}.

\section{The simplicial path object category}\label{simplicialexample}
In this final section, we fill in the details of the proof, sketched in
Section~\ref{ssetssection}, that simplicial sets form a path object category.
We begin by establishing some notational conventions. Recall that $\Delta$ is
the category whose objects are the ordered sets $[n] = \{0, \dots, n\}$ (for $n
\in \mathbb N$) and whose morphisms are order-preserving maps, and that the
category $\cat{SSet}$ of simplicial sets is the presheaf category $[\Delta^\op,
\cat{Set}]$. As before, we write $X_n \defeq X([n])$ for the set of
$n$-simplices of a simplicial set $X$; and given an element $x \in X_n$ and a
map $\alpha \colon [m] \to [n]$, we shall write $x\cdot \alpha$ for
$X(\alpha)(x) \in X_m$. We refer to maps of $\Delta$ as \emph{simplicial
operators}, and call monomorphisms \emph{face operators} and epimorphisms
\emph{degeneracy operators}. Of particular note are the operators $\delta_i
\colon [n-1] \to [n]$ and $\sigma_i \colon [n+1] \to [n]$ (for $0 \leqslant i
\leqslant n$) defined as follows: $\delta_i$ is the unique monomorphism $[n-1]
\to [n]$ whose image omits $i$; whilst $\sigma_i$ is the unique epimorphism
$[n+1] \to [n]$ whose image repeats~$i$. The maps $\delta_i$ and $\sigma_i$
generate the category $\Delta$ under composition, though not freely; they obey
the following \emph{simplicial identities}, describing how faces and
degeneracies commute past each other:
\begin{equation*}
\begin{aligned}
\delta_j \delta_i &= \delta_i \delta_{j-1} \qquad \text{for $i < j$} \\
\sigma_j \sigma_i &= \sigma_i \sigma_{j+1} \qquad\!\!\, \text{for $i \leqslant j$}
\end{aligned} \qquad
\begin{aligned}
\sigma_j \delta_i &= \begin{cases}
\delta_i \sigma_{j-1} & \text{for $i < j$} \\
\id & \text{for $i = j, j+1$} \\
\delta_{i-1} \sigma_j & \text{for $i > j+1$.}
\end{cases}
\end{aligned}
\end{equation*}

\subsection{Path objects}
Our first task is to describe the construction assigning to each simplicial set
$X$ the simplicial set of Moore paths in $X$. We begin by making precise the
informal description of $n$-dimensional Moore paths given in
Section~\ref{ssetssection}.
\begin{Defn}\label{mooren}
Let $X$ be a simplicial set and let $\xi, \xi'$ be $n$-simplices in $X$. An
\emph{$n$-dimensional Moore path} from $\xi$ to $\xi'$ is given by
$n$-simplices $\xi = \zeta_0, \dots, \zeta_k = \xi'$ and $(n+1)$-simplices
$\phi_1$, \dots, $\phi_k$, together with a function
\begin{equation}\label{traversal}\theta \colon \{1, \dots, k\} \to [n] \times
\{+,-\}\end{equation}
such that
\begin{equation}\label{discipline}
\begin{aligned}
\zeta_{i-1} &= \begin{cases}
  \phi_i \cdot \delta_{m+1} & \text{if $\theta(i) = (m,+)$}\\
  \phi_i \cdot \delta_m & \text{if $\theta(i) = (m,-)$}
\end{cases} \qquad
\zeta_i &= \begin{cases}
  \phi_i \cdot \delta_m & \text{if $\theta(i) = (m,+)$}\\
  \phi_i \cdot \delta_{m+1} & \text{if $\theta(i) = (m,-)$\ .}
\end{cases}
\end{aligned}
\end{equation}
\end{Defn}
We call functions like~\eqref{traversal} \emph{$n$-dimensional traversals}; the
natural number $k$ will be called the \emph{length} of the traversal. In order
to explain the role which such traversals play, we consider once again the
example of a $1$-dimensional Moore path from~\eqref{typ1} above:
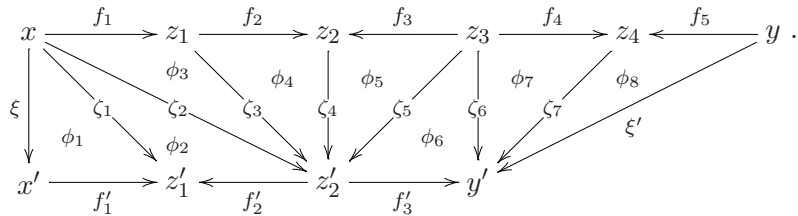
\begin{equation*}
\cd[@+1.5em]{
  x \ar[r]^{f_1} \ar[d]_\xi \ar[dr]|{\zeta_1} \ar[drr]|{\zeta_2} &
  z_1 \ar[r]^{f_2} \ar[dr]|{\zeta_3} &
  z_2 \ar[d]|{\zeta_4} &
  z_3 \ar[l]_{f_3} \ar[dl]|{\zeta_5} \ar[d]|{\zeta_6} \ar[r]^{f_4} \ar@{}[dr]|(0.3){\phi_7} &
  z_4 \ar[dl]|{\zeta_7} \ar@{}[d]|(0.3){\phi_8} &
  y\rlap{ .} \ar[l]_{f_5} \ar[dll]^{\xi'} \\
  x' \ar[r]_{f'_1} \ar@{}[ur]|(0.3){\phi_1}&
  z'_1 \ar@{}[u]|(0.25){\phi_2} \ar@{}[u]|(0.75){\phi_3} \ar@{}[ur]|(0.7){\phi_4} &
  z'_2 \ar[l]^{f'_2} \ar[r]_{f'_3} &
  y' \ar@{}[ul]|(0.7){\phi_5} \ar@{}[ul]|(0.3){\phi_6} & {}
}
\end{equation*}
The intuition is that each $2$-simplex $\phi_i$ appearing in it is so oriented
that it may be ``projected'' in an orientation-preserving manner onto $\xi$.
Under this projection, precisely one edge of $\phi_i$ will be collapsed; and
the value of $\theta(i)$ indicates at which vertex of $\xi$ this collapsing
occurs, and in which direction the collapsed edge points. Thus, for the above
example, the values of the traversal $\theta$ are given by the list
\[(1,+),\,(1,-),\,(0,+),\,(0,+),\,(0,-),\,(1,+),\,(0,+),\,(0,-)\ \text.\]
Observe that the face equations of~\eqref{discipline} enforce the discipline
that makes this intuition precise.
\begin{Prop}\label{simpln}
There is a simplicial set $MX$ whose $n$-simplices are the $n$-dimensional
Moore paths in $X$.
\end{Prop}
To prove this Proposition, we must describe a coherent action by simplicial
operators on the simplices of $MX$. The basic idea may be illustrated with
reference to the examples of Moore paths given above. In the case of a
$1$-dimensional Moore path like~\eqref{typ1}, its two $0$-dimensional faces
should be obtained by projection on to the top and bottom rows of the diagram;
whilst in the case of a $0$-dimensional Moore path like~\eqref{typ0}, its image
under the degeneracy $\sigma_0 \colon [1] \to [0]$ will be the $1$-dimensional
Moore path
\begin{equation*}
\cd[@+2em]{
  x \ar[r]^{f_1} \ar[d]_{x\cdot\sigma_0} \ar[dr]|{f_1} &
  z_1 \ar[r]^{f_2} \ar[d]|{z_1\cdot\sigma_0} \ar[dr]|{f_2}&
  z_2 \ar[d]|{z_2\cdot\sigma_0}&
  z_3 \ar[l]_{f_3} \ar[r]^{f_4} \ar[d]|{z_3\cdot\sigma_0} \ar[dl]|{f_3} \ar[dr]|{f_4}&
  z_4 \ar[d]|{z_4\cdot\sigma_0}&
  x' \ar[l]_{f_5} \ar[d]|{x'\cdot\sigma_0} \ar[dl]|{f_5}\\
  x \ar[r]_{f_1} \ar@{}[ur]|(0.28){f_1\cdot \sigma_0} \ar@{}[ur]|(0.72){f_1\cdot \sigma_1} &
  z_1 \ar[r]_{f_2} \ar@{}[ur]|(0.28){f_2\cdot \sigma_0} \ar@{}[ur]|(0.72){f_2\cdot \sigma_1} &
  z_2 &
  z_3 \ar[l]^{f_3} \ar[r]_{f_4} \ar@{}[ur]|(0.28){f_4\cdot \sigma_0} \ar@{}[ur]|(0.72){f_4\cdot \sigma_1}
  \ar@{}[ul]|(0.28){f_3\cdot \sigma_0} \ar@{}[ul]|(0.72){f_3\cdot \sigma_1} &
  z_4 &
  x'\ \text. \ar[l]^{f_5} \ar@{}[ul]|(0.28){f_5\cdot \sigma_0} \ar@{}[ul]|(0.72){f_5\cdot \sigma_1}
}
\end{equation*}

In giving a formal proof of Proposition~\ref{simpln}, it will be convenient to
consider first the special case where $X = 1$, the terminal simplicial set.
Observe that the $n$-simplices of $M1$ are simply $n$-dimensional traversals.
Given such a traversal $\theta$ of length $k$ and a simplicial operator $\alpha
\colon [m] \to [n]$, it is easy to see that there is a unique pullback diagram
\begin{equation}\label{pbsquare}
    \cd{
        \{1, \dots, \ell\} \ar[r]^-{\psi} \ar[d]_{\bar \alpha} &
        [m] \times \{+,-\} \ar[d]^{\alpha \times \{+,-\}} \\
        \{1, \dots, k\} \ar[r]_-{\theta} &
        [n] \times \{+,-\}
    }
\end{equation}
such that the following condition is satisfied:
\begin{equation}\tag{\dag}\label{condition}
\left\{\
\begin{array}{p{10cm}}
\rm{$\bar \alpha$ is order-preserving, and for $i \in \{1,\dots,k\}$, the
restriction of $\psi$ to the fibre of $\bar \alpha$ over $i$ is order-reversing
if $\theta(i) = (x,+)$ for some $x$, and order-preserving if $\theta(i) =
(x,-)$.}
\end{array}
\right.
\end{equation}
We may therefore define $\theta \cdot \alpha \in (M1)_m$ to be $\psi$. It is
immediate that $\theta \cdot \id_m = \theta$, and it's easy to check that
pullback squares satisfying (\dag) are stable under composition, so that
$\theta \cdot (\alpha \beta) = (\theta \cdot \alpha) \cdot \beta$. Thus $M1$ is
a simplicial set as required; and we now exploit this fact in proving the same
for a general $MX$. The key observation we will make is that a typical
$n$-simplex of $MX$ is given by an $n$-simplex of $M1$ that has been suitably
labelled with simplices of $X$. To make this precise, we first need:
\begin{Not}
Given a traversal $\theta$, we write
\begin{equation*}
\theta^+(i) = \left\{ \begin{array}{cl}
x & \mbox{if } \theta(i) = (x, +) \\
x+1 & \mbox{if } \theta(i) = (x, -)
\end{array} \right. \quad \text{and} \quad
\theta^-(i) = \left\{ \begin{array}{cl}
x+1 & \mbox{if } \theta(i) = (x, +) \\
x & \mbox{if } \theta(i) = (x, -)
\end{array} \right.
\end{equation*}
Also, in circumstances where it cannot cause confusion, we may choose to write
the $x$ such that $\theta(i) = (x,\rho)$ simply as $\theta(i)$.
\end{Not}
In terms of this notation, to give an $n$-simplex of $MX$ is to give a
traversal $\theta \in (M1)_n$ of length $k$ together with $n$-simplices
$\zeta_0, \dots, \zeta_k$ and $(n+1)$-simplices $\phi_1$, \dots, $\phi_k$ such
that for each $1 \leqslant i \leqslant k$, we have $\zeta_{i-1} = \phi_i \cdot
\delta_{\theta^-i}$ and $\zeta_i = \phi_i \cdot \delta_{\theta^+i}$. We may
further recast this description by means of the following definition.
\begin{Defn}\label{simplrealisation}
If $\theta \colon \{1, \dots, k\} \to [n] \times \{+,-\}$ is a traversal, then
its \emph{simplicial realisation} is the simplicial set $\hat \theta$ obtained
as the colimit of the diagram
\begin{equation}\label{dtheta}
 D_\theta \defeq \ \cd{
    [n] \ar[dr]_{\delta_{\theta^-(1)}} & & [n] \ar[dl]^{\delta_{\theta^+(1)}} \ar[dr]_{\delta_{\theta^-(2)}} & &
    \dots \ar[dl]^{\delta_{\theta^+(2)}} \ar[dr]_{\delta_{\theta^-(k)}} & & [n] \ar[dl]^{\delta_{\theta^+(k)}} \\
    & [n+1] & & [n+1] & & [n+1]}
\end{equation}
of simplicial sets, wherein we identify objects and morphisms of $\Delta$ with
their images under the Yoneda embedding $\Delta \to [\Delta^\op, \cat{Set}]$.
We write $s_\theta, t_\theta \colon [n] \to \hat \theta$ for the colimit
injections from the leftmost and rightmost copies of $[n]$.
\end{Defn}
Using this definition we see that an $n$-simplex of $MX$ may be identified with
a pair $(\theta, \phi)$ where $\theta$ is an $n$-simplex of $M1$ and $\phi
\colon \hat \theta \to X$. Consequently, to equip $MX$ with the structure of a
simplicial set, it suffices to prove:
\begin{Prop}\label{realisemaps}
To every simplicial operator $\alpha \colon [m] \to [n]$ and traversal $\theta$
we may assign a map of simplicial sets $\hat \alpha \colon \widehat{\theta
\cdot \alpha} \to \hat \theta$. Moreover, this assignment is functorial in the
sense that $\widehat{1_{[n]}} = 1_{\hat \theta}$ and $\widehat{\beta \circ
\alpha} = \hat \beta \circ \hat \alpha$.
\end{Prop}
Indeed, given this result we may define an action of the simplicial operator
$\alpha \colon [m] \to [n]$ on an $n$-simplex $(\theta, \phi)$ of $MX$ by
$(\theta, \phi) \cdot \alpha = (\theta \cdot \alpha, \phi \circ \hat \alpha)$;
with the coherence equations for this action following easily from the
functoriality in Proposition~\ref{realisemaps}. Thus to complete the proof that
$MX$ is a simplicial set, it remains only to give:

\begin{proof}[of Proposition~\ref{realisemaps}]
Suppose that $\theta \colon \{1, \dots, k\} \to [n] \times \{+,-\}$, and let us
write $\psi \colon \{1, \dots, \ell\} \to [m] \times \{+,-\}$ for the traversal
$\theta \cdot \alpha$; recall that it is defined as the unique function fitting
into a pullback square of the form~\eqref{pbsquare}. For each $1 \leqslant j
\leqslant k$, we define $\psi_i$ to be the traversal given by restricting
$\psi$ to the fibre $\bar \alpha^{-1}(i) \subset \{1, \dots, \ell\}$ (and
renumbering, since the smallest element of this fibre is probably not $1$).
Then to give a morphism $\hat \psi \to \hat \theta$, it suffices to find maps
$f_1, \dots, f_k$ rendering commutative the diagram
\begin{equation}\label{f1fk}
  \cd[@R-0.5em]{
    [m] \ar[d]_{\alpha} \ar[dr]^{s_{\psi_1}} & &
    [m] \ar[dl]_{t_{\psi_1}} \ar[dr]^{s_{\psi_2}} \ar[d]^{\alpha} & &
    \dots \ar[dl]_{t_{\psi_2}} \ar[dr]^{s_{\psi_k}}  & &
    [m] \ar[dl]_{t_{\psi_k}} \ar[d]^{\alpha} \\
    [n] \ar[dr]_{\delta_{\theta^-(1)}} &
    \widehat{\psi_1} \ar@{.>}[d]_{f_1}&
    [n] \ar[dl]^{\delta_{\theta^+(1)}} \ar[dr]_{\delta_{\theta^-(2)}} &
    \widehat{\psi_2} \ar@{.>}[d]_{f_2}&
    \dots \ar[dl]^{\delta_{\theta^+(2)}} \ar[dr]_{\delta_{\theta^-(k)}} &
    \widehat{\psi_k} \ar@{.>}[d]_{f_k}&
    [n]\rlap{ ,} \ar[dl]^{\delta_{\theta^+(k)}} \\
    & [n+1] & & [n+1] & & [n+1]}
\end{equation}
since $\hat \psi$ and $\hat \theta$ are the respective colimits of the upper
and lower zigzags in this diagram. In fact, we claim that there is a unique way
of choosing the $f_i$'s given the remaining data: which as well as proving the
existence of $\hat \alpha$, easily implies the functoriality in $\alpha$. It
remains to prove the claim. So let $1 \leqslant i \leqslant k$, and assume that
$\theta(i) = (a, \mathord -)$ for some $0 \leqslant a \leqslant n$: something
we may do without loss of generality, since the case where $\theta(i) =
(a,\mathord +)$ is entirely dual. We first consider the situation in which $a$
is not in the image of $\alpha$. In this case, we have that $\widehat{\psi_i} =
[m]$ and $s_{\psi_i} = t_{\psi_i} = 1_{[m]}$, so that $f_i$ is necessarily
unique, and will exist so long as the square\vskip-\baselineskip
\begin{equation*}
    \cd[@-1em]{
     & [m] \ar[dl]_\alpha \ar[dr]^{\alpha} \\ [n] \ar[dr]_{\delta_a} & & [n] \ar[dl]^{\delta_{a+1}} \\ & [n+1]
    }
\end{equation*}
commutes: which it does since $a$ is not in the image of $\alpha$. Turning now
to the case where $a$ \emph{is} in the image of $\alpha$, we observe that
$\alpha^{-1}(a)$ is a non-empty segment of $[m]$, and so writing $p$ and $q$
for its smallest and largest elements, we have by virtue of the condition
(\dag) defining $\psi$ that $\psi_i$ is the traversal of length $q-p+1$ given
by $\psi_i(x) = (p+x-1, \mathord -)$; 
so that to give $f_i$ is equally well to give the dotted maps in
\begin{equation}\label{togive}
\cd[@R-0.8em@C-0.5em]{
    [m] \ar[dr]^{\delta_{p}} \ar@/_36pt/[dddrrr]_{h_p \defeq \delta_a \alpha}  & &
    [m] \ar[dl]_{\delta_{p+1}} \ar[dr]^{\delta_{p+1}} \ar@/_18pt/@{.>}[dddr]|{h_{p+1}} & &
    \dots \ar[dl]_{\delta_{p+2}} \ar[dr]^{\delta_{q}}  & &
    [m] \rlap{ .}\ar[dl]_{\delta_{q+1}} \ar@/^36pt/[dddlll]^{\ \ h_{q+1} \defeq \delta_{a+1} \alpha}  \\
    & [m+1] \ar@/_9pt/@{.>}[ddrr]|{g_p} & & [m+1] \ar@{.>}[dd]|{g_{p+1}} & & [m+1] \ar@/^9pt/@{.>}[ddll]|{g_q}\\ \\
    & & & [n+1]
}
\end{equation}
We wish to show that there is a unique way of doing this. Note first that if
such maps exist, then commutativity in
\begin{equation*}
    \cd[@-1em]{
     & [m-1] \ar[dl]_{\delta_p} \ar[dr]^{\delta_p} \\ [m] \ar[dr]_{\delta_p} & & [m] \ar[dl]^{\delta_{p+1}} \\ & [m+1]
    }
\end{equation*}
implies that $h_p.\delta_p = h_{p+1}.\delta_p$; in other words, that $h_p$ and
$h_{p+1}$ are the same, except possibly for their value at $p$. More generally,
we see that for each $p \leqslant j \leqslant q$, the maps $h_j$ and $h_{j+1}$
agree everywhere except possibly at $j$. It follows that a typical $h_j$ must
agree with $h_p$ at all values except for those in $\{p, \dots, j-1\}$, and
must agree with $h_{q+1}$ at all values except for those in $\{j, \dots, q\}$.
Since $h_p$ and $h_{q+1}$ agree at all values except $\{p, \dots, q\}$, this is
certainly possible, and forces the definition
\begin{equation}\label{hdef}
    h_{j}(x) \defeq \begin{cases}
    h_{q+1}(x) & \text{for $x < j$\ ;} \\
    h_p(x) & \text{for $x \geqslant j$\ ,}
    \end{cases} \qquad \text{i.e.,}\quad
    h_{j}(x) \defeq \begin{cases}
    \alpha(x) & \text{for $x < j$\ ;} \\
    \alpha(x)+1 & \text{for $x \geqslant j$\ .}
    \end{cases}
\end{equation}
Now since $g_j.\delta_j = h_j$ and $g_j.\delta_{j+1} = h_{j+1}$, it is easy to
see that this in turn forces the definition
\begin{equation}\label{gdef}
    g_{j}(x) \defeq \begin{cases}
    \alpha(x) & \text{for $x \leqslant j$\ ;} \\
    \alpha(x-1)+1 & \text{for $x > j$\ .}
    \end{cases}
\end{equation}
Thus we have shown that there a unique way of filling in the dotted arrows
in~\eqref{togive}, and hence also a unique way of filling in the dotted arrows
in~\eqref{f1fk}: which completes the construction of the desired map $\hat
\alpha$, and also implies the functoriality of the construction in $\alpha$.
\end{proof}

\subsection{First axiom} We have now completed the construction of the
simplicial set $MX$ of Moore paths in $X$. Recall that we did so by making use
of the isomorphism
\begin{equation}\label{pra}
    (MX)_n \cong \sum_{\theta \in (M1)_n} \cat{SSet}(\hat \theta, X)
\end{equation}
to define the action of a simplicial operator $\alpha \colon [m] \to [n]$ on an
$n$-simplex $(\theta, \phi)$ by $(\theta, \phi) \cdot \alpha = (\theta \cdot
\alpha, \phi \circ \hat \alpha)$. We will exploit the isomorphism~\eqref{pra}
further in proving that:
\begin{Prop}\label{ax1ss}
The simplicial sets $MX$ provide data for Axiom 1 of a path object structure on
$\cat{SSet}$.
\end{Prop}
\begin{proof}
Given a map $f \colon X \to Y$ of simplicial sets, we define the action of $Mf$
on an $n$-simplex $(\theta, \phi)$ by $Mf(\theta, \phi) = (\theta, f \phi)$. It
is immediate that this yields a map of simplicial sets $MX \to MY$,
functorially in $f$. To see that the functor $M$ so defined preserves
pullbacks, observe that it suffices to do so componentwise; and
that~\eqref{pra} expresses each such component $(M\thg)_n$ as a coproduct of
limit-preserving functors, and so as a pullback- (and indeed, connected limit-)
preserving functor.

We define the maps $s_X, t_X \colon MX \to X$ by sending $(\theta, \phi) \in
(MX)_n$ to the $n$-simplex classified by the composites $\phi \circ s_\theta$
and $\phi \circ t_\theta \colon [n] \to X$ (where $s_\theta$ and $t_\theta$ are
as in Definition~\ref{simplrealisation}). Commutativity in the extremal squares
of~\eqref{f1fk} ensure that $s_X$ and $t_X$ are maps of simplicial sets. To
define $r_X \colon X \to MX$, we note that the realisation of the empty
$n$-dimensional traversal $\epsilon$ is simply $[n]$, so that we may define
$r_X(x) = (\epsilon, \bar x)$, where $\bar x \colon [n] \to X$ is the map
classifying $x$. Compatibility with the simplicial structure again follows
from~\eqref{f1fk}.

To define the morphism $m_X \colon MX \times_X MX \to MX$, observe that an
$n$-simplex of the domain of this map is given by elements $(\theta_1, \phi_1)$
and $(\theta_2, \phi_2) \in (MX)_n$ with $\phi_1 \circ t_{\theta_1} = \phi_2
\circ s_{\theta_2}$. We take their composite to be $(\theta_1 + \theta_2,
\psi)$, where if $\theta_1$ and $\theta_2$ are traversals of length $k_1$ and
$k_2$ respectively, then $\theta_1 + \theta_2$ is the traversal of length $k_1
+ k_2$ given by:
\begin{equation*}
    (\theta_1 + \theta_2)(x) = \begin{cases}
    \theta_1(x) & \text{for $1 \leqslant x \leqslant k_1$ ;}\\
    \theta_2(x-k_1) & \text{for $k_1 < x \leqslant k_1 + k_2$ .}
\end{cases}
\end{equation*}
To give $\psi$, we observe that we have a pushout square:
\begin{equation*}
    \cd[@-1em]{
     & [n] \ar[dl]_{t_{\theta_1}} \ar[dr]^{s_{\theta_2}} \\ \hat{\theta_1} \ar[dr]_{} & & \hat{\theta_2} \ar[dl]^{} \\ & \widehat{\theta_1 + \theta_2}
    }
\end{equation*}
where the two unlabelled arrows are the canonical inclusions, so that we may
take $\psi \colon \widehat{\theta_1 + \theta_2} \to X$ to be the map induced by
the universal property of this pushout applied to the pair $\phi_1 \colon
\hat{\theta_1} \to X$ and $\phi_2 \colon \hat{\theta_2} \to X$. Simpliciality
of $m_X$ follows by observing that its induced effect on diagrams of the
form~\eqref{f1fk} is simply that of by placing them side by side.

It is easy to see that the data given so far equip $MX \rightrightarrows X$
with the structure of an internal category, naturally in $X$; and so it remains
only to provide the identity-on-objects involution $\tau_X \colon MX \to MX$.
To do so, we first define the \emph{reverse} of a traversal $\theta \colon \{1,
\dots, k\} \to [n] \times \{+,-\}$. This will be the traversal $\theta^o$ of
length $k$ given by
\begin{equation*}
    \theta^o(x) = \begin{cases} (y,+) & \text{if $\theta(k+1-x) = (y,-)$;} \\
    (y,-) & \text{if $\theta(k+1-x) = (y,+)$.}
    \end{cases}
\end{equation*}
Now observe that there is a canonical isomorphism $e_\theta \colon
\widehat{\theta^o} \to \hat \theta$, since the diagram $D_{\theta^o}$ of which
the latter is a colimit is obtained by laterally mirroring the diagram
$D_\theta$ for the former. We may therefore define the involution $\tau_X$ by
sending the $n$-simplex $(\theta, \phi)$ to $(\theta^o, \phi \circ e_\theta$).
Since $e_{\theta^o} = e_\theta^{-1}$, this operation is involutive, and is
easily seen to respect the category structure and to be natural in $X$. It
remains to check that $\tau_X$ commutes with the action of simplicial
operators. Observe first that for any simplicial operator $\alpha \colon [m]
\to [n]$, we have $(\theta \cdot \alpha)^o = \theta^o \cdot \alpha$ by virtue
of the condition (\dag) in the definition of the action of $\alpha$. Moreover,
any diagram of the form
\begin{equation*}
    \cd[@C+0.8em]{
      \widehat{\theta \cdot \alpha} \ar[r]^-{e_{\theta \cdot \alpha}} \ar[d]_{\hat \alpha} &
      \widehat{\theta^o \cdot \alpha} \ar[d]^{\hat \alpha} \\
      \hat{\theta} \ar[r]_-{e_{\theta}} &
      \widehat{\theta^o}
    }
\end{equation*}
will commute, since the effect of the operation $(\thg)^o$ on diagrams of the
form~\eqref{f1fk} is to mirror them laterally. It follows from this that
$\tau_X$ commutes with the action by simplicial operators as required.
\end{proof}

\subsection{Second axiom} We have now provided all the data for Axiom 1 of a
path object structure on the category of simplicial sets; and so now turn to
Axiom 2.
\begin{Prop}\label{ax2ss}
The endofunctor $M$ on simplicial sets defined in Proposition~\ref{ax2ss} may
be equipped with a strength which validates Axiom 2 for a path object category.
\end{Prop}
\begin{proof}
By the remarks made in Section~\ref{strengthsection}, it suffices to provide
components $\alpha_{1,X} \colon M1 \times X \to MX$ for the strength. To do
this, we first define, for every $n$-dimensional traversal $\theta$, a map
$j_\theta \colon \hat \theta \to [n]$ induced by the cocone
\begin{equation}\label{coconediag}
\cd[@R-1em]{
    [n] \ar[dr]^{\delta_{\theta^-(1)}} \ar@/_42pt/[dddrrr]_{\id} & &
    [n] \ar[dl]_{\delta_{\theta^+(1)}} \ar[dr]^{\delta_{\theta^-(2)}} \ar@/_18pt/[dddr]_{\id} & &
    \dots \ar[dl]_{\delta_{\theta^+(2)}} \ar[dr]^{\delta_{\theta^-(k)}}  & &
    [n] \ar[dl]_{\delta_{\theta^+(k)}} \ar@/^42pt/[dddlll]^{\id} \\
    & [n+1] \ar@/_18pt/[ddrr]|{\sigma_{\theta(1)}} & &
    [n+1] \ar[dd]|{\sigma_{\theta(2)}} & &
    [n+1] \ar@/^18pt/[ddll]|{\sigma_{\theta(k)}}\\ \\
    & & & [n]
}
\end{equation}
under the diagram $D_\theta$ of which $\hat \theta$ is a colimit. Now given an
$n$-simplex $(\theta, x)$ of $M1 \times X$, we define its image under
$\alpha_{1,X}$ to be $(\theta, \bar x \circ j_\theta)$, where as before, $\bar
x \colon [n] \to X$ is the map classifying $x$. To show that this yields a map
of simplicial sets, it is evidently enough to verify that diagrams of the form
\begin{equation*}
    \cd{
      \widehat{\theta \cdot \alpha} \ar[r]^-{j_{\theta \cdot \alpha}} \ar[d]_{\hat \alpha} &
      [m] \ar[d]^{\alpha} \\
      \hat{\theta} \ar[r]_-{j_{\theta}} &
      [n]
    }
\end{equation*}
commute. Now since $\hat \alpha$ is defined by the diagram~\eqref{f1fk}, it
suffices for this to show that, using the notation of that diagram, the square
\begin{equation*}
\cd[@R-0.7em]{
   \hat{\psi_i} \ar[r]^{j_{\psi_i}} \ar[d]_{f_i} & [m] \ar[d]^{\alpha} \\
   [n+1] \ar[r]_-{\sigma_{\theta(i)}} & [n]
}
\end{equation*}
commutes for each $i$. Recalling that the morphism $f_i$ was defined by the
diagram~\eqref{togive}---in which we have assumed, without loss of generality,
that $\theta(i) = (a,-)$---it therefore suffices to show, using the notation of
that diagram, that $\sigma_a.g_j = \alpha.\sigma_j$ for each $p \leqslant j
\leqslant q$: and this follows by direct examination of the
equation~\eqref{gdef} defining $g_j$. This completes the verification that
$\alpha_{1,X}$ is a map of simplicial sets; and it is straightforward to verify
that these maps are natural in $X$, and that the strength axioms are satisfied.
Now that $r$ is strong follows by observing that $j_\epsilon = \id_{[n]}$
(where as before $\epsilon$ is the empty traversal of dimension $n$); that $s$
and $t$ are strong follows from the commutativity of the extremal triangles
in~\eqref{coconediag}; that $m$ is strong follows by noting that for
$n$-dimensional traversals $\theta_1$ and $\theta_2$, the following diagram
commutes:
\begin{equation*}
    \cd[@+0.5em]{
        \hat{\theta_1} \ar[r] \ar[dr]_{j_{\theta_1}}& \widehat{\theta_1 + \theta_2} \ar[d]|{j_{\theta_1 + \theta_2}} & \hat{\theta_2} \ar[l] \ar[dl]^{j_{\theta_2}}\\
        & [n]
    }
\end{equation*}
in which the horizontal arrows are the canonical inclusions; whilst that $\tau$
is strong follows from the fact that for any traversal $\theta$, the diagram
\begin{equation*}
    \cd[@+0.5em]{
        \hat{\theta} \ar[r]^{e_\theta} \ar[dr]_{j_{\theta}}& \widehat{\theta^o} \ar[d]^{j_{\theta^o}}\\
        & [n]
    }
\end{equation*}
commutes.
\end{proof}
\subsection{Third axiom} The last part of the proof that simplicial sets form a
path object category will be to construct the maps $\eta_X \colon MX \to MMX$
which, to every $n$-dimensional Moore path in $X$, associate an $n$-dimensional
Moore path in $MX$ which contracts the given path on to its endpoint.
We motivate the construction by considering the example~\eqref{contractsample}
of Section~\ref{ssetssection} above, which demonstrates the action of $\eta_X$
on the $0$-dimensional Moore path of~\eqref{typ1}.
The first observation we can make about the Moore path
of~\eqref{contractsample} is that it has the same underlying traversal as the
path~\eqref{typ1} of which it is a contraction. So for a general
$n$-dimensional Moore path $x = (\theta, \phi)$ of length $k$ in $X$, we aim to
give $\eta_X(x)$ of the form $(\theta, \eta_X(\phi))$ in $MX$. To do this, we
must provide $(k+1)$ $n$-simplices and $k$ $(n+1)$-simplices of $MX$ which are
matched together in the fashion dictated by $\theta$. To give the $n$-simplices
is straightforward; as suggested by our example~\eqref{contractsample}, these
will be obtained by suitably truncating the original path.
\begin{Defn}
Let $x = (\theta, \phi)$ be an $n$-dimensional Moore path of length $k$ in $X$.
For $0 \leqslant i \leqslant k$, we define the \emph{tail at $i$} to be the
$n$-dimensional Moore path $x^i = (\theta^i, \phi^i)$ of length $k-i$ whose
traversal is the function $x \mapsto \theta(x+i)$, and whose second component
is the composite
\begin{equation*}
  \hat{\theta^i} \to \hat \theta \xrightarrow{\phi} X
\end{equation*}
in which the first arrow is the canonical inclusion.
\end{Defn}
It remains to define the $(n+1)$-simplices of $MX$ that will mediate between
the tail at $i$ and the tail at $i+1$; and consideration of the
example~\eqref{contractsample} suggests that these should be obtained by first
forming a suitable degeneracy of the tail at $i$, and then removing its first
element. For example, looking at the case $i=1$ in~\eqref{contractsample}, we
see that we have:
\begin{equation*} \cd[@+1em]{
  & z_1 \ar[r]|{f_1} \ar@{.>}[dl]_{z_2 \cdot \sigma_0} \ar[d]_{f_1} &
  z_2 \ar[dl]|{z_2 \cdot \sigma_0} &
  z_3 \ar[l]|{f_2} \ar[r]|{f_3} \ar[dll]|{f_2} \ar[dl]|{z_3 \cdot \sigma_0} \ar[d]|{f_3} &
  z_4 \ar[dl]|{z_4\cdot \sigma_0} &
  x' \ar[l]|{f_4} \ar[dll]|{f_4} \ar[dl]^{x'\cdot \sigma_0} \\
  z_1 \ar@{.>}[r]|{f_1} &
  z_2 &
  z_3 \ar[l]|{f_2} \ar[r]|{f_3} &
  z_4 &
  x'\rlap{ ,} \ar[l]|{f_4}
}
\end{equation*}
which has been obtained by removing the first element from the degeneracy
$\sigma_0$ of the tail at $1$. This motivates the following definition:
\begin{Defn}\label{etadefnsimpl}
Let $x = (\theta, \phi)$ be an $n$-dimensional Moore path of length $k$ in $X$.
We define $\eta_X(x)$ to be the $n$-dimensional Moore path in $MX$ whose
traversal is $\theta$, and whose second component $\eta_X(\phi) \colon \hat
\theta \to MX$ is the map induced by the cocone
\begin{equation}\label{etatheta}
\cd[@R-1em]{
    [n] \ar[dr]^{\delta_{\theta^-(1)}} \ar@/_42pt/[dddrrr]_{x^0} & &
    [n] \ar[dl]_{\delta_{\theta^+(1)}} \ar[dr]^{\delta_{\theta^-(2)}} \ar@/_18pt/[dddr]_{x^1} & &
    \dots \ar[dl]_{\delta_{\theta^+(2)}} \ar[dr]^{\delta_{\theta^-(k)}}  & &
    [n]\rlap{ .} \ar[dl]_{\delta_{\theta^+(k)}} \ar@/^42pt/[dddlll]^{x^k} \\
    & [n+1] \ar@/_18pt/[ddrr]|{(x^0 \cdot \sigma_{\theta(1)})^1} & &
    [n+1] \ar[dd]|{(x^1 \cdot \sigma_{\theta(2)})^1} & &
    [n+1] \ar@/^18pt/[ddll]|{(x^{k-1} \cdot \sigma_{\theta(k)})^1}\\ \\
    & & & MX
}
\end{equation}
%
%
%
%
\end{Defn}
In order for this definition to make sense, we must verify that the
diagram~\eqref{etatheta} appearing in it is commutative. We do so using the
following result.
\begin{Lemma}\label{shift}
Let $x = (\theta, \phi)$ be an $n$-dimensional Moore path of length $k$, and
let $\alpha \colon [m] \to [n]$ be a simplicial operator. Then for any $0
\leqslant i \leqslant k$, we have $x^i \cdot \alpha = (x \cdot \alpha)^{\bar
\imath}$, where $\bar \imath = \sum_{j=1}^i \abs{\alpha^{-1}(\theta(j))}$.
\end{Lemma}
\begin{proof}
Observe that as well as $x^i$, the tail at $i$, we may also by duality form
$x_i$, the \emph{head at~$i$}: we take $x_i \defeq ((x^o)^{k-i})^o$. It is now
easy to see that $x = m_X(x^i,x_i)$; and moreover, that if $x = m_X(z,y)$ with
$y$ of length $i$, then necessarily $y = x_i$ and $z = x^i$. Now we have that $
x \cdot \alpha = m_X(x^i,x_i) \cdot \alpha = m_X(x^i \cdot \alpha, x_i \cdot
\alpha)$ and hence that $x^i \cdot \alpha = (x \cdot \alpha)^{\bar \imath}$,
where $\bar \imath$ is the length of $x_i \cdot \alpha$. But examination of the
pullback square~\eqref{pbsquare} shows that this length is $\sum_{j=1}^i
\abs{\alpha^{-1}(\theta(j))}$ as required.
\end{proof}
We now show commutativity in~\eqref{etatheta}. For $1 \leqslant i \leqslant k$,
we have by condition (\dag) on page~\pageref{condition} that the first element
of the traversal associated with $x^{i-1} \cdot \sigma_{\theta(i)}$ is of the
form $(\theta^-(i), \rho)$: and so by the preceding Lemma, we obtain
\begin{equation}\label{obtained}
\begin{aligned}    (x^{i-1} \cdot \sigma_{\theta(i)})^1 \cdot \delta_{\theta^-(i)} & = (x^{i-1} \cdot
\sigma_{\theta(i)} \cdot \delta_{\theta^-(i)})^0 = x^{i-1}\\
\text{and} \ \
    (x^{i-1} \cdot \sigma_{\theta(i)})^1 \cdot \delta_{\theta^+(i)} & = (x^{i-1} \cdot
\sigma_{\theta(i)} \cdot \delta_{\theta^+(i)})^1 = (x^{i-1})^1 = x^i\rlap{ .}
\end{aligned}
\end{equation}
\begin{Prop}
The assignation $x \mapsto \eta_X(x)$ of Definition~\ref{etadefnsimpl} provides
data for an instance of Axiom 3 for a path object structure on the category of
simplicial sets.
\end{Prop}
\begin{proof}
The hardest part of the proof will be to verify that $\eta_X$ is a map of
simplicial sets. Given $x \in (MX)_n$ and a simplicial operator $\alpha \colon
[m] \to [n]$, we must show that $\eta_X(x \cdot \alpha) = \eta_X(x) \cdot
\alpha$. Suppose that $x = (\theta, \phi)$ is a traversal of length $k$, and
write $\psi = \theta \cdot \alpha$. Then to verify that $\eta_X(x \cdot \alpha)
= \eta_X(x) \cdot \alpha$ it suffices to show that for each $1 \leqslant i
\leqslant k$, the square
\begin{equation}\label{idiag}
    \cd[@R-0.5em@C+3em]{
        \hat \psi_i \ar[d]_{f_i} \ar[r] & \hat \psi \ar[d]^{\eta_X(x \cdot \alpha)} \\
        [n+1] \ar[r]_-{(x^{i-1} \cdot \sigma_{\theta(i)})^1} & MX
    }
\end{equation}
commutes, where $f_i$ and $\psi_i$ are as defined in~\eqref{f1fk}, and where
the unlabelled horizontal map is the evident inclusion. So let us fix some $1
\leqslant i \leqslant k$, and suppose without loss of generality that
$\theta(i) = (a,\mathord -)$. Let us also define
$    r = \sum_{j=1}^{i-1} \left|\alpha^{-1}(\theta(j))\right|$,
and observe that with this definition, the canonical inclusion $\hat \psi_i \to
\hat \psi$ maps the $j$th $m$- or $(m+1)$-simplex of $\hat \psi_i$ to the
$(j+r)$th $m$- or $(m+1)$-simplex of $\hat \psi$. To prove that~\eqref{idiag}
commutes, we first consider the degenerate case where $\psi_i$ is an empty
traversal: that is, when $a$ is not in the image of $\alpha$. Then by the proof
of Proposition~\ref{realisemaps}, we have $\hat \psi_i = [m]$ and $f_i =
\delta_a \alpha$, so that we must verify that the diagram
\begin{equation*}
    \cd[@R-0.5em@C+2em]{
        [m] \ar[d]_{\delta_a \alpha} \ar[dr]^{(x \cdot \alpha)^r} \\
        [n+1] \ar[r]_-{(x^{i-1} \cdot \sigma_{a})^1} & MX
    }
\end{equation*}
commutes. But by~\eqref{obtained} and Lemma~\ref{shift} we have $(x^{i-1} \cdot
\sigma_a)^1 \cdot \delta_a \cdot \alpha = x^{i-1} \cdot \alpha = (x \cdot
\alpha)^r$ as required. Suppose now that $\psi_i$ is a non-empty traversal,
given as in the proof of Proposition~\ref{realisemaps} by the function
$\psi_i(x) = (p+x-1, \mathord -)$.
 Now to
show commutativity in~\eqref{idiag}, it suffices to do so on precomposition
with the $q-p+1$ colimit injections $[m+1] \to \hat \psi_i$. If we label these
injections with natural numbers $p \leqslant j \leqslant q$, then
precomposing~\eqref{idiag} with the $j$th one yields the diagram
\begin{equation}\label{tocheckhard}
    \cd[@R-0.5em@C+1.5em]{
        [m+1] \ar[d]_{g_{j}} \ar[dr]^{\ \ \ \ ((x\cdot \alpha)^{j+r-p} \cdot \sigma_{j})^1} \\
        [n+1] \ar[r]_-{(x^{i-1} \cdot \sigma_a)^1} & MX\ \text.
    }
\end{equation}
Now the first element of the traversal associated with $(x^{i-1} \cdot
\sigma_a)$ is $(a, -)$; and hence by Lemma~\ref{shift}, we have $(x^{i-1} \cdot
\sigma_a)^1 \cdot g_{j} = (x^{i-1} \cdot \sigma_a \cdot g_j)^{|g_j^{-1}(a)|}$.
But by direct examination of~\eqref{gdef}, we see that $|g_j^{-1}(a)| = j - p +
1$; moreover, we have as above that $x^{i-1} \cdot \sigma_a \cdot g_j = x^{i-1}
\cdot \alpha \cdot \sigma_j = (x \cdot \alpha)^r \cdot \sigma_j$, and so
conclude that the map along the lower side of~\eqref{tocheckhard} is equal to
$(x^{i-1} \cdot \sigma_a)^1 \cdot g_{j} =  ((x \cdot \alpha)^r \cdot
\sigma_j)^{j - p + 1} $. Now, along the upper side we have $((x\cdot
\alpha)^{j+r-p} \cdot \sigma_{j})^1 = (((x\cdot \alpha)^r)^{j-p} \cdot
\sigma_{j})^1 = ((x\cdot \alpha)^r \cdot \sigma_j)^{\bar \imath + 1}$, where
here $\bar \imath = \sum_{h = 1}^{j-p} | \sigma_j^{-1}(\psi(h+r))|$. To prove
equality in~\eqref{tocheckhard}, it therefore suffices to show that $\bar
\imath = j-p$. But we have that $\psi(h+r) = \psi_i(h) = p + h -1$ so that
$\bar \imath = \sum_{h = 1}^{j-p} | \sigma_j^{-1}(p + h - 1)| = j-p$ as
required. This proves that~\eqref{tocheckhard}, and hence~\eqref{idiag}, are
commutative, which completes the verification that $\eta_X$ is a map of
simplicial sets.  It is now straightforward to verify that the $\eta_X$'s so
defined are natural in $X$ and strong; and so it remains only to check the
equations~\eqref{seta}--\eqref{tri2}.

For equations~\eqref{seta} and~\eqref{teta}, it is immediate
from~\eqref{etatheta} that $s_{MX}(\eta_X(x)) = x^0 = x$ and that
$t_{MX}(\eta_X(x)) = x^k = r_X(t_X(x))$ as required, whilst for
equation~\eqref{tri2}, we calculate that $\eta_X(r_X(x)) = \eta_X(\epsilon, x)
= (\epsilon, (\epsilon, x)) = r_{MX}(r_X(x))$ as required. Thus it remains to
verify equations~\eqref{Ms} and~\eqref{Mt}, which, we recall, say that
$
    Ms_{X} .\eta_X = \id_{MX}$ and $
    Mt_{X} .\eta_X = \alpha_{1,X} . (M!, t_X)$.
Now, given an $n$-simplex $x = (\theta, \phi)$ of $X$ we have that
$Ms_x(\eta_X(x))$ and $Mt_X(\eta_X(x))$ are given by the simplices $(\theta,
s_X.\eta_X(\phi))$ and $(\theta, t_X.\eta_X(\phi))$ respectively, and that
$\alpha_{1,X} . (M!, t_X)$ is given by the simplex $(\theta, t_X(x).j_\theta)$,
where $j_\theta$ is as defined in Proposition~\ref{ax2ss}. Thus to verify
equations~\eqref{Ms} and~\eqref{Mt}, we must prove that
\begin{equation*}
    s_X.\eta_X(\phi) = \phi \qquad \text{and} \qquad
    t_X.\eta_X(\phi) = t_X(x).j_\theta
\end{equation*}
as maps $\hat \theta \to X$. But since $\hat \theta$ is the colimit of the
diagram~\eqref{dtheta}, it suffices to show these two equalities on
precomposition with the each of the colimit injections $q_1, \dots, q_k \colon
[n+1] \to \hat \theta$. The latter case is simpler: we calculate that
\begin{align*}
t_X.\eta_X(\phi).q_i &=
t_X.((x^{i-1} \cdot \sigma_{\theta(i)})^1)
= t_X.(x^{i-1} \cdot
\sigma_{\theta(i)}) \\
&= t_X(x^{i-1}) \cdot \sigma_{\theta(i)} = t_X(x) .
\sigma_{\theta(i)} = t_X(x).j_\theta.q_i
\end{align*}
as required. For the former case, we must show that $s_X((x^{i-1} \cdot
\sigma_{\theta(i)})^1) = \phi.q_i$ as maps $[n+1] \to X$. For definiteness, let
us suppose that $\theta(i) = (a, \mathord -)$; the case where it is $(a,
\mathord +)$ is entirely dual. Now $s_X((x^{i-1} \cdot \sigma_{\theta(i)})^1)$
is the composite
\begin{equation}\label{scomp}
    [n+1] \xrightarrow{u} \widehat{\theta^{i-1} \cdot
\sigma_{a}} \xrightarrow{\widehat{\sigma_{a}}} \widehat{\theta^{i-1}} \to \widehat{\theta} \xrightarrow{\phi} X
\end{equation}
in which the first arrow $u$ picks out the second copy of $[n+1]$ in the
colimit defining its codomain. In order to simplify this composite further,
consider the morphism $\widehat{\sigma_{a}}$ appearing in it. This is induced
by a diagram of the form~\eqref{f1fk}, whose left-hand edge is given by
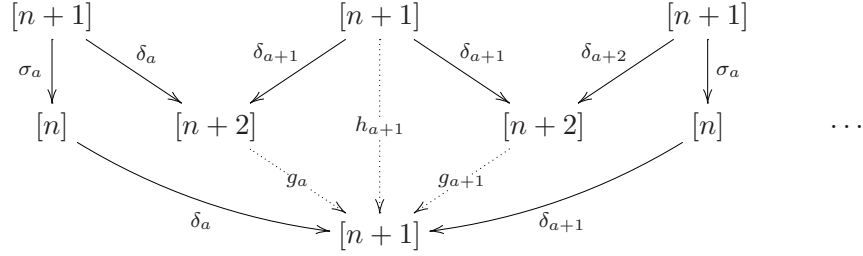
\begin{equation*}
  \cd{
    [n+1] \ar[d]_{\sigma_{a}} \ar[dr]^{\delta_a} & &
    [n+1] \ar[dl]_{\delta_{a+1}} \ar[dr]^{\delta_{a+1}} \ar@{.>}[dd]|{h_{a+1}} & &
    [n+1] \ar[d]^{\sigma_{a}} \ar[dl]_{\delta_{a+2}} \\
    [n] \ar@/_9pt/[drr]_{\delta_{a}} &
    [n+2] \ar@{.>}[dr]|{g_a} &
    &
    [n+2] \ar@{.>}[dl]|{g_{a+1}}&
    [n] \ar@/^9pt/[dll]^{\delta_{a+1}} & \dots\\
    & & [n+1]}
\end{equation*}
where the dotted arrows are defined as in~\eqref{togive}. In particular, by
inspection of~\eqref{hdef}, we see that $h_{a+1}$ is the identity. Hence the
composite $\widehat{\sigma_a} u$ is the morphism picking out the leftmost
$[n+1]$-simplex in the colimit defining $\widehat \theta^{i-1}$, and
hence~\eqref{scomp} is equal to $\phi.q_i$ as required.
\end{proof}
This completes the verification of Axiom 3, and hence of:
\begin{Prop}
The category of simplicial sets bears the structure of a path object category.
\end{Prop}

\bibliographystyle{acm}
\bibliography{rhgg2}

\begin{thebibliography}{10}

\bibitem{Awodey2008Homotopy}
{\sc Awodey, S., and Warren, M.}
\newblock Homotopy theoretic models of identity types.
\newblock {\em Mathematical Proceedings of the Cambridge Philosophical Society
  146}, 1 (2009), 45--55.

\bibitem{Bousfield1977Constructions}
{\sc Bousfield, A.}
\newblock Constructions of factorization systems in categories.
\newblock {\em Journal of Pure and Applied Algebra 9}, 2-3 (1977), 207--220.

\bibitem{Cartmell1986Generalised}
{\sc Cartmell, J.}
\newblock Generalised algebraic theories and contextual categories.
\newblock {\em Annals of Pure and Applied Logic 32\/} (1986), 209--243.

\bibitem{Gambino2008identity}
{\sc Gambino, N., and Garner, R.}
\newblock The identity type weak factorisation system.
\newblock {\em Theoretical Computer Science 409\/} (2008), 94--109.

\bibitem{Garner2008Two-dimensional}
{\sc Garner, R.}
\newblock Two-dimensional models of type theory.
\newblock {\em Mathematical Structures in Computer Science 19}, 4 (2009),
  687--736.

\bibitem{Garner2008Types}
{\sc Garner, R., and Berg, B. v.~d.}
\newblock Types are weak {$\omega$}-groupoids.
\newblock {\em Proceedings of the London Mathematical Society\/} (2010).
\newblock In press.

\bibitem{Hofmann1995interpretation}
{\sc Hofmann, M.}
\newblock On the interpretation of type theory in locally cartesian closed
  categories.
\newblock In {\em Computer science logic (Kazimierz, 1994)}, vol.~933 of {\em
  Lecture Notes in Computer Science}. Springer, 1995, pp.~427--441.

\bibitem{Hofmann1998groupoid}
{\sc Hofmann, M., and Streicher, T.}
\newblock The groupoid interpretation of type theory.
\newblock In {\em Twenty-five years of constructive type theory (Venice,
  1995)}, vol.~36 of {\em Oxford Logic Guides}. Oxford University Press, 1998,
  pp.~83--111.

\bibitem{Hyland1982effective}
{\sc Hyland, J.~M.}
\newblock The effective topos.
\newblock In {\em The L.E.J. Brouwer Centenary Symposium (Noordwijkerhout,
  1981)}, vol.~110 of {\em Studies in Logic and the Foundations of
  Mathematics}. North-Holland, 1982, pp.~165--216.

\bibitem{Jacobs1993Comprehension}
{\sc Jacobs, B.}
\newblock Comprehension categories and the semantics of type dependency.
\newblock {\em Theoretical Computer Science 107}, 2 (1993), 169--207.

\bibitem{Johnstone1997Cartesian}
{\sc Johnstone, P.}
\newblock Cartesian monads on toposes.
\newblock {\em Journal of Pure and Applied Algebra 116\/} (1997), 199--220.

\bibitem{Kock1970Monads}
{\sc Kock, A.}
\newblock Monads on symmetric monoidal closed categories.
\newblock {\em Archiv der Mathematik 21}, 1 (1970), 1--10.

\bibitem{Lumsdaine2009Weak}
{\sc Lumsdaine, P.~L.}
\newblock Weak {$\omega$}-categories from intensional type theory.
\newblock In {\em Typed Lambda Calculi and Applications}, no.~5608 in Lecture
  Notes in Computer Science. Springer, 2009.

\bibitem{Nordstrom1990Programming}
{\sc Nordstr\"{o}m, B., Petersson, K., and Smith, J.}
\newblock {\em Programming in {M}artin-{L}\"{o}f's Type Theory}, vol.~7 of {\em
  International Series of Monographs on Computer Science}.
\newblock Oxford University Press, 1990.

\bibitem{Oosten2010Notion}
{\sc Oosten, J.~v.}
\newblock A notion of homotopy for the effective topos, 2010.
\newblock Unpublished note, available at
  \url{http://www.staff.science.uu.nl/~ooste110}.

\bibitem{Pitts2000Categorical}
{\sc Pitts, A.~M.}
\newblock Categorical logic.
\newblock In {\em Handbook of Logic in Computer Science, Volume 5. Algebraic
  and Logical Structures}, S.~Abramsky, D.~M. Gabbay, and T.~S.~E. Maibaum,
  Eds. Oxford University Press, 2000, pp.~39--128.

\bibitem{Quillen1967Homotopical}
{\sc Quillen, D.~G.}
\newblock {\em Homotopical algebra}, vol.~43 of {\em Lecture Notes in
  Mathematics}.
\newblock Springer-Verlag, 1967.

\bibitem{Rosick'y2002Lax}
{\sc Rosick{\'y}, J., and Tholen, W.}
\newblock Lax factorization algebras.
\newblock {\em Journal of Pure and Applied Algebra 175\/} (2002), 355--382.

\bibitem{Street1996Categorical}
{\sc Street, R.}
\newblock Categorical structures.
\newblock In {\em Handbook of algebra, {V}ol.\ 1}. North-Holland, 1996,
  pp.~529--577.

\bibitem{Warren2008Homotopy}
{\sc Warren, M.}
\newblock {\em Homotopy theoretic aspects of constructive type theory}.
\newblock PhD thesis, Carnegie Mellon University, 2008.

\bibitem{Warren2009characterization}
{\sc Warren, M.}
\newblock A characterization of representable intervals.
\newblock Preprint, 2009.
\newblock Available at \url{http://arxiv.org/abs/0903.3743}.

\end{thebibliography}

 \end{document}